\documentclass[12pt]{amsart}
\usepackage[all]{xy}
\usepackage{graphicx}
\usepackage{tikz}  

\usepackage{amsmath}
\usepackage{amssymb}
\usepackage{amsfonts}
\usepackage{mathrsfs}
\usepackage{mathtools}

\newcommand{\Cdb}{\ensuremath{\mathbb{C}}}

\newcommand{\Ndb}{\ensuremath{\mathbb{N}}}

\newcommand{\Rdb}{\ensuremath{\mathbb{R}}}

\newcommand{\Zdb}{\ensuremath{\mathbb{Z}}}

\newcommand{\A}{\mbox{${\mathcal A}$}}

\newcommand{\E}{\mbox{${\mathcal E}$}}
\newcommand{\F}{\mbox{${\mathcal F}$}}

\renewcommand{\O}{\mbox{${\mathcal O}$}}

\newcommand{\R}{\mbox{${\mathcal R}$}}
\renewcommand{\S}{\mbox{${\mathcal S}$}}

\newcommand{\norm}[1]{\Vert#1\Vert}
\newcommand{\bignorm}[1]{\bigl\Vert#1\bigr\Vert}
\newcommand{\Bignorm}[1]{\Bigl\Vert#1\Bigr\Vert}

\newtheorem{theorem}{Theorem}[section]
\newtheorem{lemma}[theorem]{Lemma}
\newtheorem{corollary}[theorem]{Corollary}
\newtheorem{proposition}[theorem]{Proposition}
\newtheorem{definition}[theorem]{Definition}
\theoremstyle{remark}
\newtheorem{remark}[theorem]{\bf Remark}
\theoremstyle{definition}
\newtheorem{example}[theorem]{\bf Example}
\numberwithin{equation}{section}

\begin{document}

\title[Multivariable $H^\infty$ functional calculus]{New properties of the multivariable $H^\infty$ functional calculus of
sectorial operators}

\author[O. Arrigoni]{Olivier Arrigoni}
\email{olivier.arrigoni@univ-fcomte.fr}
\author[C. Le Merdy]{Christian Le Merdy}
\email{clemerdy@univ-fcomte.fr}
\address{Laboratoire de Math\'ematiques de Besan\c con, UMR 6623, 
CNRS, Universit\'e Bourgogne Franche-Comt\'e,
25030 Besan\c{c}on Cedex, FRANCE}

\date{\today}

\maketitle

\begin{abstract}
This paper is devoted to the multivariable
$H^\infty$ functional calculus associated with 
a finite commuting family of sectorial operators on Banach space. 
First we prove that if $(A_1,\ldots, A_d)$ is such a family,
if $A_k$ is $R$-sectorial of
$R$-type $\omega_k\in(0,\pi)$, $k=1,\ldots,d$, and if 
$(A_1,\ldots, A_d)$ admits a bounded $H^\infty(\Sigma_{\theta_1}\times
\cdots\times\Sigma_{\theta_d})$ joint functional calculus 
for some $\theta_k\in (\omega_k,\pi)$, then it 
admits a bounded $H^\infty(\Sigma_{\theta_1}\times
\cdots\times\Sigma_{\theta_d})$ joint functional calculus 
for all $\theta_k\in (\omega_k,\pi)$, $k=1,\ldots,d$. Second
we introduce square functions adapted to the multivariable case
and extend to this 
setting some of the well-known one-variable results
relating the boundedness of $H^\infty$ functional calculus 
to square function estimates. Third, on $K$-convex reflexive spaces,
we establish sharp dilation properties for 
$d$-tuples $(A_1,\ldots, A_d)$ which
admit a bounded $H^\infty(\Sigma_{\theta_1}\times
\cdots\times\Sigma_{\theta_d})$ joint functional calculus 
for some $\theta_k<\frac{\pi}{2}$.
\end{abstract}

\vskip 1cm
\noindent
{\it 2000 Mathematics Subject Classification :} 
47A60, 47A20, 47D03.

\vskip 1cm

\section{Introduction}

This paper deals with the $H^\infty$ functional calculus of finite 
families of commuting sectorial operators on Banach space. 
This multivariable form of the classical (=single valued) $H^\infty$ functional calculus was introduced by 
D. Albrecht in \cite{Al} on $L^p$-spaces. In this setting, more results 
were obtained in \cite{FMI} shortly after Albrecht's thesis. 
Later on in \cite{LLLM} and \cite{KW1}, 
$H^\infty$ functional calculus for couples of sectorial operators was 
investigated on general Banach spaces, with applications 
to abstract maximal regularity.

A new approach to the multivariable $H^\infty$ functional 
calculus was recently 
given by the authors in \cite{ArLM}. In particular we proved that 
if $(A_1,\ldots, A_d)$ is a finite family 
of commuting sectorial operators on a Banach space $X$ with property 
$(\alpha)$, and if each $A_k$ admits a bounded 
$H^\infty$ functional calculus, then
$(A_1,\ldots, A_d)$ admits a bounded $H^\infty$ joint functional calculus
(the case $d=2$ goes back to \cite{LLLM}).
Further, in the specific case when $X=L^p(\Omega)$ for $1<p<\infty$, 
we showed that 
$(A_1,\ldots, A_d)$ admits a bounded $H^\infty(\Sigma_{\theta_1}\times
\cdots\times\Sigma_{\theta_d})$ joint functional calculus 
for some $\theta_k<\frac{\pi}{2}$, $k=1,\ldots d$, 
if and only if
$(A_1,\ldots, A_d)$ can be dilated into a commuting $d$-tuple  
$(B_1,\ldots, B_d)$ of sectorial operators of types
$<\frac{\pi}{2}$ on some $L^p(\Omega')$,
such that $e^{-tB_k}$ is a positive contraction for any $k=1,\ldots,d$ and any 
$t\geq 0$ (the case $d=1$ goes back to \cite{AFLM}).

The present paper is devoted to new properties of the multivariable
$H^\infty$ functional calculus. We work with families
$(A_1,\ldots, A_d)$ acting on large classes of Banach spaces,
so that a bounded $H^\infty$ functional calculus
for each $A_k$ does not ensure that $(A_1,\ldots, A_d)$ 
admits a bounded $H^\infty$ joint functional calculus.
Our purpose is two-fold. On the one hand we extend an important result of 
N. Kalton and L. Weis on the optimal angle of bounded
$H^\infty$ functional calculus. It is shown in 
\cite[Proposition 5.1]{KW1} that if $A$ is an 
$R$-sectorial operator of $R$-type
$\omega\in(0,\pi)$ and if $A$ admits a bounded 
$H^\infty(\Sigma_\theta)$ functional calculus
for some $\theta\in(\omega,\pi)$, then it actually
admits a bounded $H^\infty(\Sigma_\theta)$ functional calculus
for all $\theta\in(\omega,\pi)$.  Theorem \ref{d-KW} below provides 
a $d$-variable version of this result. The proof
relies on the so-called Franks-McIntoch decomposition 
of analytic functions on sectors, which
was introduced in \cite{FMI}. In the
case $d=1$, the proof of Theorem \ref{d-KW}
yields a new proof of \cite[Proposition 5.1]{KW1}.

On the other hand, under the mild assumption that
$X$ is reflexive and $K$-convex, 
we show in Theorem \ref{d-FW1} that 
if $(A_1,\ldots,A_d)$ admits a bounded  
$H^\infty(\Sigma_{\theta_1}\times
\cdots\times\Sigma_{\theta_d})$ joint functional calculus 
for some $\theta_k<\frac{\pi}{2}$, $k=1,\ldots,d$, then
it can be dilated into a commuting $d$-tuple  
$(B_1,\ldots, B_d)$ of sectorial operators 
on a Bochner space $L^2(\Omega;X)$,
such that each $B_k$ generates a bounded $c_0$-group and 
$(B_1,\ldots, B_d)$ admits a bounded 
$H^\infty(\Sigma_{\frac{\pi}{2}}\times
\cdots\times\Sigma_{\frac{\pi}{2}})$ joint functional calculus.
We also investigate the converse issue of how to 
deduce a bounded $H^\infty$ joint functional calculus from a dilation property. This set of results is given in Section 5. It
aims at generalizing most of the remarkable results 
of A. Fr\"ohlich and L. Weis (see \cite{FW}) established in the case $d=1$. 

In the classical $H^\infty$ functional calculus theory,
square functions play a prominent role. We refer the
reader e.g. to \cite{CDMY,Survey,KW2,HH,BOOK}
and the references therein for information. 
In Section 4, we introduce square functions associated 
to a commuting family of sectorial opertors and extend to this 
setting some of the well-known one-variable results
relating the boundedness of $H^\infty$ functional calculus 
to square function estimates. The approach is
similar to the one-variable case, however these results 
are important because they are used in the proofs of the dilation
results presented above.

\section{Preliminaries, functional calculus and $R$-boundedness}\label{S2}
We first give some general notations and conventions which will be used 
along the paper. All our Banach spaces $X$ will
be complex ones. We let $B(X)$ denote the Banach algebra of all
bounded operators on $X$, and we let $I_X$ denote the
identity operator on $X$.

Let $A$ be a possibly unbounded operator on $X$.
We let $D(A)$, $N(A)$ and $R(A)$ denote its domain, its kernel and its
range, respectively. We let $\sigma(A)$ denote the spectrum of $A$
and for any $z\in\Cdb\setminus\sigma(A)$, we let 
$R(z,A) = (z I_X-A)^{-1}$ denote the resolvent operator. 

For any measure space $(\Omega,dm)$ and any $1\leq p<\infty$, 
we let $L^p(\Omega;X)$ denote the 
Bochner space of all (classes of) measurable functions
$\zeta\colon\Omega\to X$ such that $\norm{\zeta(\,\cdotp)}_X$ belongs to
$L^p(\Omega)$. This is a Banach space for the norm
$$
\norm{\zeta}_{L^p(\Omega;X)}\,=\,\Bigl(
\int_\Omega\norm{\zeta(t)}_X^p\, dm(t)\Bigr)^{\frac{1}{p}}.
$$
See e.g. \cite{DU} for some background.

Whenever $\O$ is a set, $X$ is a Banach space and $f\colon \O\to
X$ is a bounded 
function, we set
$$
\norm{f}_{\infty,\O}\,=\,\sup
\bigl\{\norm{f(\lambda)}_X\, :\, \lambda\in\O\bigr\}.
$$
If $\O\subset \Cdb^d$ is a non empty open set for some integer $d\geq 1$,
we let $H^\infty(\O;X)$ denote the Banach space of all
bounded holomorphic functions from $\O$ into $X$,
equipped with $\norm{\,\cdotp}_{\infty,\O}$. When
$X=\Cdb$, this space is simply denoted by $H^\infty(\O)$.
This is a Banach algebra.

We will assume that the reader is familiar with the basics 
of semigroup theory. We refer e.g. to \cite{Go,Paz} for information.

If $\varphi_1,\varphi_2$ are numerical functions acting on a set $\E$,
we will write  that ``$\varphi_2(e)\lesssim\varphi_1(e)$ for 
$e\in\E$" if 
there exists a constant $K>0$ such that $\varphi_2(e)\leq K\varphi_1(e)$
for any $e\in\E$. Further we will write $\varphi_2(e)\approx\varphi_1(e)$
if we both have $\varphi_2(e)\lesssim\varphi_1(e)$ and
$\varphi_1(e)\lesssim\varphi_2(e)$ for 
$e\in\E$.

Following \cite{ArLM}, we now 
recall the definitions and some features of 
the $H^\infty$ functional calculus of a commuting finite family of sectorial operators.

For any $\theta \in (0,\pi)$, we let
$$
\Sigma_\theta = \left\lbrace z \in \mathbb{C}^* 
\, :\, \left|\text{Arg}(z) \right| < \theta \right\rbrace
$$
be the open sector of angle $2\theta$ around the positive real line.

Let $X$ be a Banach space and let $A\colon D(A)\to X$
be a closed and densely defined operator on $X$. We say that $A$ is
sectorial of type $\omega \in (0,\pi)$ if its spectrum
satisfies $\sigma(A) \subset \overline{\Sigma_{\omega}}$ 
and for any $\theta$ in $(\omega,\pi)$, 
there exists a constant $C_\theta\geq 0$ such that
\begin{equation}\label{sectoriel}
\left\| zR(z,A) \right\| \leq C_\theta, \qquad z \in \mathbb{C} \setminus \overline{\Sigma_{\theta}}.
\end{equation}
It follows from the Laplace formula that if $-A$ is 
the generator of a bounded $C_0$-semigroup on $X$, then $A$ is sectorial 
of type $\frac{\pi}{2}$. Further
$A$ is sectorial of type $\omega < \frac{\pi}{2}$ 
if and only if $-A$ is the generator of a bounded 
analytic semigroup.

We fix an integer $d\geq 1$ and consider a $d$-tuple
$(A_1,\ldots,A_d)$ of commuting sectorial operators
on $X$, with respective types $\omega_1,\ldots,\omega_d$.

Let $\theta_1,\ldots,\theta_d$ in $(0,\pi)$.
For any subset 
$\Lambda\subset \left\lbrace 1,\ldots,d \right\rbrace $, 
we let $H^\infty_0 \bigl(\prod_{k \in \Lambda} 
\Sigma_{\theta_k} \bigr)$ denote 
the subalgebra of $H^\infty (\Sigma_{\theta_1} 
\times \cdots \times 
\Sigma_{\theta_d})$ of all bounded holomorphic
functions $f\colon \Sigma_{\theta_1} 
\times \cdots \times 
\Sigma_{\theta_d}\to\Cdb\,$ depending only on the variables 
$(z_k)_{k\in \Lambda}$, such that there exist a constant  
$C\geq 0$ and a family $(s_k)_{k \in \Lambda}$ of positive
real numbers satisfying
\begin{equation} \label{Hinfini0}
\vert f(z_1,\ldots,z_d) \vert \leq \,
C\,\prod_{k \in \Lambda} \dfrac{|z_k|^{s_k}}{(1+|z_k|)^{2s_k}}, 
\qquad (z_k)_{k\in \Lambda} \in \prod_{k \in \Lambda} 
\Sigma_{\theta_k}.
\end{equation}

Assume now that $\theta_k\in(\omega_k,\pi)$ and let 
$\nu_k\in(\omega_k,\theta_k)$, for any $k=1,\ldots,d$. 
Given any $f$ in $H_0^{\infty}\bigl(\prod_{k\in\Lambda} 
\Sigma_{\theta_k}\bigr)$, with 
$\Lambda \subset \left\lbrace 1,\ldots,d \right\rbrace$, 
$\Lambda \neq \emptyset$, we set
\begin{equation}\label{fAi}
f(A_1,\ldots,A_d) =\Bigl(\frac{1}{2\pi i} \Bigr)^{|\Lambda|} 
\int_{\prod_{k \in \Lambda} \partial \Sigma_{\nu_k}} 
f(z_1,\ldots,z_d) \prod_{k \in \Lambda} R(z_k,A_k) \prod_{k \in \Lambda} dz_k,
\end{equation}
where the boundaries $\partial \Sigma_{\nu_k}$ are 
oriented counterclockwise. 
The assumptions (\ref{sectoriel}) and (\ref{Hinfini0}) 
ensure that this integral is absolutely 
convergent and defines an element of $B(X)$. Further 
by Cauchy's Theorem, this definition does not depend on 
the choice of the $\nu_k$, $1\leq k\leq d$. 

When $J=\emptyset$, $H^\infty_0 \bigl(\prod_{k \in 
\emptyset} \Sigma_{\theta_k} \bigr)$ 
is equal to the space of constant functions on 
$\Sigma_{\theta_1} \times \cdots \times \Sigma_{\theta_d}$. 
For such a function $f \equiv a$ (here $a\in\Cdb$),
we set $f(A_1,\ldots,A_d) = a I_X$.

It turns out that the sum of the spaces
$H^\infty_0 \bigl(\prod_{k \in \Lambda} 
\Sigma_{\theta_k} \bigr)$, for $\Lambda\subset\{1,\ldots,d\}$, 
is a direct one. Then this sum 
\begin{equation}\label{Hinfini01}
H_{0,1}^\infty(\Sigma_{\theta_1} \times 
\cdots \times \Sigma_{\theta_d}) : = 
\bigoplus_{\Lambda \subset 
\left\lbrace 1,\ldots,d \right\rbrace} H^\infty_0 
\Bigl( \prod_{k \in \Lambda} \Sigma_{\theta_k} \Bigr)
\end{equation}
is a subalgebra of $H^\infty (\Sigma_{\theta_1} 
\times \cdots \times 
\Sigma_{\theta_d})$.

For any function 
$$
f = \sum_{\Lambda \subset \left\lbrace 1,\ldots,d 
\right\rbrace} f_\Lambda\ \in
H_{0,1}^\infty(\Sigma_{\theta_1} 
\times \cdots \times \Sigma_{\theta_d}),
$$ 
with $f_\Lambda\in H^\infty_0 \left(\prod_{k \in \Lambda} 
\Sigma_{\theta_k} \right)$, we set
$$
f(A_1,\ldots,A_d) = \sum_{\Lambda\subset 
\left\lbrace 1,\ldots,d \right\rbrace} f_\Lambda(A_1,\ldots,A_d).
$$
Then the functional calculus mapping $f\mapsto
f(A_1,\ldots,A_d)$ from 
$H_{0,1}^\infty(\Sigma_{\theta_1} \times \cdots 
\times \Sigma_{\theta_d})$ into $B(X)$ 
is an algebra homomorphism (see \cite[Lemma 2.1]{ArLM}).

\begin{definition}\label{Bdd-FC1}
We say that $(A_1,\ldots,A_d)$ admits an 
$H^\infty(\Sigma_{\theta_1} \times 
\cdots \times \Sigma_{\theta_d})$ joint functional calculus if 
there exists a constant $K >0$ such that 
\begin{equation}\label{Bdd-FC}
\left\| f(A_1,\ldots,A_d) \right\| \leq K \left\| f \right\|_{\infty, 
\Sigma_{\theta_1} \times \cdots \times \Sigma_{\theta_d}},\qquad
f\in H_{0,1}^{\infty}(\Sigma_{\theta_1} \times \cdots\times \Sigma_{\theta_d}).
\end{equation}
\end{definition}

In the above definition it would be tempting to talk about a ``bounded 
$H^\infty(\Sigma_{\theta_1} \times 
\cdots \times \Sigma_{\theta_d})$ joint functional calculus" as we did in 
the Introduction. However for convenience we choose to drop the word
``bounded" from now on. Also when $d=1$, we naturally
drop the word ``joint" and simply say that a sectorial operator 
$A$ admits an $H^\infty(\Sigma_\theta)$ functional calculus
if we have an estimate $\norm{f(A)}\lesssim\norm{f}_{\infty,\Sigma_\theta}$,
for $f\in H^\infty_0(\Sigma_\theta)$.

It will be convenient to extend Definition \ref{Bdd-FC1} to the limit case when 
$\theta_k$ is replaced by $\omega_k$, as follows.

\begin{definition}\label{Bdd-FC2}
We say that $(A_1,\ldots,A_d)$ admits an 
$H^\infty(\Sigma_{\omega_1}\times\cdots\times\Sigma_{\omega_d})$ 
joint functional calculus if it admits an 
$H^\infty(\Sigma_{\theta_1}\times\cdots\times\Sigma_{\theta_d})$ 
joint functional calculus for any $\theta_k\in(\omega_k,\pi)$,
and there exists a constant $K\geq 0$ such that
(\ref{Bdd-FC}) holds true for any $\theta_k\in(\omega_k,\pi)$,
$k=1,\ldots,d$.
\end{definition}

For any integer $m\geq 1$, we set
\begin{equation}\label{Phi_m}
\Phi_m(z) = \,\frac{m^2 z}{(m+z)(1+mz)}\,, \quad z\in\Cdb\setminus\Rdb_-.
\end{equation}
A key property of this sequence is that whenever $A$ is a sectorial
operator, we have
\begin{equation}\label{Approx}
\forall\, x\in \overline{R(A)},\qquad \lim_{m}\Phi_m(A)x\,=x,
\end{equation}
see e.g. \cite[2.C]{Survey}. 
We show below that this approximation property allows 
to simplify the definition of $H^\infty$ joint functional calculus
in the case when all the $A_k$ have a dense range 
(in this case, $H^\infty_{0,1}$ is replaced by $H^\infty_0$
in the estimate (\ref{Bdd-FC})).

\begin{lemma}\label{DenseRange}
Assume that $\overline{R(A_k)}=X$ for all $k=1,\ldots,d$.
If there exists a constant $K>0$ such that
\begin{equation}\label{Bdd-FC3}
\left\| f(A_1,\ldots,A_d) \right\| \leq K \left\| f \right\|_{\infty, 
\Sigma_{\theta_1} \times \cdots \times \Sigma_{\theta_d}}
\end{equation}
for any
$f\in H_{0}^{\infty}(\Sigma_{\theta_1} \times \cdots\times \Sigma_{\theta_d})$,
then $(A_1,\ldots,A_d)$ admits an 
$H^\infty(\Sigma_{\theta_1}\times\cdots\times\Sigma_{\theta_d})$ 
joint functional calculus.
\end{lemma}

\begin{proof}
Let $f\in H_{0,1}^{\infty}(\Sigma_{\theta_1} 
\times \cdots\times \Sigma_{\theta_d})$.
For any $m\geq 1$, let 
\begin{equation}\label{Phi_m^d}
\Phi_m^{\otimes d}=\Phi_m\otimes\cdots\otimes \Phi_m
\end{equation} be the element
of $H^\infty_{0}(\Sigma_{\theta_1}\times\cdots\times\Sigma_{\theta_d})$
defined by 
$\Phi_m^{\otimes d}(z_1,\ldots,z_d)=\Phi_m(z_1)\cdots\Phi_m(z_d)$.
Then the product $f\Phi_m^{\otimes d}$ belongs to 
$H_{0}^{\infty}(\Sigma_{\theta_1} \times \cdots\times \Sigma_{\theta_d})$
and 
$$
(f\Phi_m^{\otimes d})(A_1,\ldots,A_d)=f(A_1,\ldots,A_d)\Phi_m(A_1)\cdots
\Phi_m(A_d).
$$ 
Applying (\ref{Bdd-FC3}) to this function we obtain that
$$
\left\| f(A_1,\ldots,A_d)\Phi_m(A_1)\cdots
\Phi_m(A_d)\right\| \leq K \Bigl(
\prod_{k=1}^d\norm{\Phi_m}_{\infty,\Sigma_{\theta_k}}\Bigr)\,\norm{f}_{\infty,
\Sigma_{\theta_1} \times \cdots \times \Sigma_{\theta_d}}.
$$
The dense range assumption and (\ref{Approx}) ensure that 
$$
f(A_1,\ldots,A_d)\Phi_m(A_1)\cdots
\Phi_m(A_d)\,\longrightarrow f(A_1,\ldots,A_d)
$$
strongly, as $m\to\infty$. 
Moreover $(\Phi_m)_{m\geq 1}$ is uniformly bounded on all sectors
$\Sigma_\theta$. Hence the above inequality yields an estimate
$$
\left\| f(A_1,\ldots,A_d)\right\| \lesssim \norm{f}_{\infty,
\Sigma_{\theta_1} \times \cdots \times \Sigma_{\theta_d}}.
$$
\end{proof}

It follows from Definition \ref{Bdd-FC1} that if 
$(A_1,\ldots,A_d)$ admits an $H^\infty(\Sigma_{\theta_1} \times 
\cdots \times \Sigma_{\theta_d})$ joint 
functional calculus, then every subfamily $(A_k)_{k \in \Lambda}$, 
with $\emptyset\not=
\Lambda \subset \left\lbrace 1,\ldots,d \right\rbrace$,
admits an 
$H^\infty( \prod_{k \in \Lambda} \Sigma_{\theta_k})$ 
joint functional calculus. 
In particular, for every $k=1,\ldots,d$, the operator
$A_k$ admits an $H^\infty(\Sigma_{\theta_k})$ 
functional calculus. The converse does not hold true. 
Indeed it follows from \cite[Theorem 3.9]{LLLM} that for 
any $1\leq p\not=2\leq\infty$, 
there exists a  commuting
couple $(A_1,A_2)$ of sectorial operators on the Schatten space $S^p$ 
such that $A_k$ admits an $H^\infty(\Sigma_{\theta_k})$ 
functional calculus for any $\theta_k\in (0,\pi)$, $k=1,2$, 
but $(A_1,A_2)$ has no joint functional calculus (see Example \ref{Ex}
for more on this).
We proved however in \cite[Section 3]{ArLM} that if $X$ has property
$(\alpha)$ or if $X$ is a Banach lattice, then 
$(A_1,\ldots,A_d)$ admits an $H^\infty(\Sigma_{\theta_1} \times \cdots \times \Sigma_{\theta_d})$ joint functional calculus provided 
that for any $k=1,\ldots,d$, $A_k$ admits an $H^\infty(\Sigma_{\theta'_k})$ 
functional calculus for
some $\theta'_k<\theta_k$. We refer the reader 
to \cite{Pi} (see also \cite{LLLM}) for the definition and information on property 
$(\alpha)$, and also to \cite[Section 7.5]{BOOK} where
it is called ``Pisier's contraction property".

We recall that if 
$X$ is reflexive, then for any $k=1,\ldots,d$,
we have a direct sum decomposition
\begin{equation}\label{Ergodic1}
X=\overline{R(A_k)}\oplus N(A_k),
\end{equation}
see \cite[Section 10.1]{BOOK}. 
We call it ``ergodic decomposition" in the sequel.

For any $\Lambda\subset\{1,\ldots,d\}$, set 
\begin{equation}\label{X-Lambda}
X_\Lambda = \biggl(\bigcap_{k\in \Lambda} \overline{R(A_k)}\biggr)\bigcap
\biggl(\bigcap_{k\notin \Lambda} N(A_k)\biggr).
\end{equation}

\begin{lemma}\label{Ergodic2}
If $X$ is reflexive, we have a direct sum decomposition 
$$
X = \bigoplus_{\Lambda\subset\{1,\ldots,d\}} X_\Lambda.
$$
\end{lemma}

\begin{proof}
We start from the decomposition $X=\overline{R(A_1)}\oplus N(A_1)$
given by (\ref{Ergodic1}).
Since $A_2$ commutes with $A_1$, the subspaces $\overline{R(A_1)}$
and $N(A_1)$ are both $A_2$-invariant. Thus we have restrictions
$$
A_{2,1}\colon D(A_2)\cap \overline{R(A_1)}\longrightarrow
\overline{R(A_1)}
\qquad\hbox{and}\qquad 
A_{2,2}\colon D(A_2)\cap N(A_1)\longrightarrow
N(A_1)
$$
of $A_2$, which 
are sectorial operators. Since $\overline{R(A_1)}$
and $N(A_1)$ are reflexive, $A_{2,1}$ and $A_{2,2}$ admit 
ergodic decompositions. Since 
$$
\overline{R(A_{2,1})} = \overline{R(A_2})\cap \overline{R(A_1)}
\qquad\hbox{and}\qquad 
N(A_{2,1})=N(A_2)\cap \overline{R(A_1)},
$$
we obtain that 
$$
\overline{R(A_1)}=\bigl(N(A_2)\cap \overline{R(A_1)}\bigr)
\oplus\bigl(\overline{R(A_2})\cap \overline{R(A_1)}\bigr).
$$
Likewise, 
$$
N(A_1)=\bigl(N(A_2)\cap N(A_1)\bigr)
\oplus\bigl(\overline{R(A_2})\cap N(A_1)\bigr).
$$
Gluing these two ergodic decompositions together, 
we obtain the result in the case $d=2$.
The general case follows by induction.
\end{proof}

\bigskip
We now turn to some background on Rademacher/Gaussian averages
and the key notion of $R$-boundedness.

Let $I$ be a countable set and let 
$(r_j)_{j \in I}$ be an independent family of Rademacher 
variables on some probability 
space $(\S,\mathbb{P})$. For any finitely supported family 
$(x_j)_{j\in I}$  in
$X$, we set
$$
\Bignorm{\sum_{j\in I} r_j \otimes x_j}_{\text{Rad}(X)}
\,=\,
\biggl(\int_{\S}\Bignorm{\sum_{j\in I} r_j(t) 
\, x_j}_X^2 d\mathbb{P}(t) \biggr)^\frac{1}{2}.
$$
This is the norm of the (finite) sum
$\sum_{j\in I} r_j \otimes x_j$ in the
Bochner space $L^2(\S;X)$.

Consider the usual case $I=\Ndb^*$.
We say that a subset $\F\subset B(X)$ is $R$-bounded if there
exists a constant $K\geq 0$ such that 
$$
\Bignorm{\sum_{j\geq 1} r_j \otimes T_j(x_j)}_{\text{Rad}(X)}
\,\leq K
\Bignorm{\sum_{j\geq 1} r_j \otimes x_j}_{\text{Rad}(X)}
$$
for any finitely supported sequences $(T_j)_{j\geq 1}$ in $\F$ and  
$(x_j)_{j\geq 1}$ in $X$. In this case, we let $\R(\F)$ be the smallest 
possible $K\geq 0$ satisfying this property; this is called 
the $R$-bounded constant of
$\F$.

Next if $A$ is a sectorial operator on $X$, we say that $A$ is  
$R$-sectorial of $R$-type $\omega \in (0,\pi)$
if $\sigma(A) \subset \overline{\Sigma_{\omega}}$ 
and for any $\theta$ in $(\omega,\pi)$, the set
$$
\bigl\{zR(z,A) \, :\, z \in \mathbb{C} 
\setminus \overline{\Sigma_{\theta}}\bigr\}\,\subset\, B(X)
$$
is $R$-bounded. We refer e.g. to \cite[Chapter 8]{BOOK} for information
and references.

There are similar definitions with Gaussian variables
replacing the Rademacher variables, as follows. 
Again let $I$ be a countable set and let 
$(g_j)_{j \in I}$ be a independent family of complex valued
standard Gaussian variables on 
$(\S,\mathbb{P})$. Then for any finitely supported family 
$(x_j)_{j\in I}$  in
$X$, we let
$$
\Bignorm{\sum_{j\in I} g_j \otimes x_j}_{G(X)}
\,=\,
\biggl(\int_{\S}\Bignorm{\sum_{j\in I} g_j(t) 
\, x_j}_X^2 d\mathbb{P}(t) \biggr)^\frac{1}{2}.
$$

We will use the classical notions of type and cotype, 
for which we refer
to \cite[Chapter 11]{DJT} (see also \cite{M}). We recall 
that a Banach lattice has finite cotype if and only if 
it has property $(\alpha)$; this class includes $L^p$-spaces
for any $1\leq p<\infty$. Further any UMD Banach space (see \cite{Bu})
has finite cotype; this class includes non commutative $L^p$-spaces
for any $1<p<\infty$. Lastly, finite cotype passes to subspaces.

When $X$ has finite cotype, then we have an equivalence
\begin{equation}\label{Equiv}
\Bignorm{\sum_{j\in I} g_j \otimes x_j}_{G(X)}\,\approx
\Bignorm{\sum_{j\in I} r_j \otimes x_j}_{\text{Rad}(X)},
\end{equation}
for finitely supported families
$(x_j)_{j\in I}$  in
$X$. We refer e.g. to \cite[Corollary 7.2.10]{BOOK}
for this result. Despite
this property, we will need to use both 
Gaussian and Rademacher averages, even on Banach
spaces with finite cotype.
The main advantage of Rademacher 
variables is their uniform boundedness, whereas 
the main advantage of Gaussian averages if the
following invariance property (see e.g. \cite[Proposition 6.1.23]{BOOK}):
for any
matrix $[a_{ij}]\in M_n(\Cdb)$ and for any
$x_1,\ldots, x_n$ in $X$, 
\begin{equation}\label{Rot}
\Bignorm{\sum_{i,j=1}^n a_{ij} g_i \otimes x_j}_{G(X)}
\,\leq\,\norm{[a_{ij}]}_{B(\ell^2_n)}\,
\Bignorm{\sum_{j=1}^n g_j \otimes x_j}_{G(X)}.
\end{equation}

\section{Angle reduction}\label{Angle}

Let $(A_1,\ldots,A_d)$ be a $d$-tuple of commuting sectorial 
operators on some Banach space $X$. It is plain to see that 
if  $(A_1,\ldots,A_d)$ admits an
$H^\infty(\Sigma_{\theta_1}\times\cdots
\times\Sigma_{\theta_d})$ joint functional
calculus, then it also admits an
$H^\infty(\Sigma_{\theta'_1}\times\cdots
\times\Sigma_{\theta'_d})$ joint functional
calculus whenever $\theta'_k\geq \theta_k$, $k=1,\ldots,d$.
The purpose of this section
is to prove the following theorem, 
which allows to consider the case when 
$\theta'_k < \theta_k$.

\begin{theorem}\label{d-KW}
Let $X$ be a Banach space and
let $(A_1,\ldots,A_d)$ be a $d$-tuple of commuting $R$-sectorial 
operators on $X$, of respective $R$-types
$\omega_1, \ldots,\omega_d$. Let $\theta_k\in(\omega_k,\pi)$, for $k=1,\ldots, d$,
and assume that $(A_1,\ldots,A_d)$ admits an
$H^\infty(\Sigma_{\theta_1}\times\cdots\times\Sigma_{\theta_d})$ joint functional
calculus. Then for any $\theta'_k\in(\omega_k,\pi)$, $k=1,\ldots, d$, the
family $(A_1,\ldots,A_d)$ admits an
$H^\infty(\Sigma_{\theta'_1}\times\cdots\times\Sigma_{\theta'_d})$ joint functional
calculus. 
\end{theorem}

This theorem is a multivariable version of a widely used 
result due to Kalton-Weis \cite[Proposition 5.1]{KW1}.
Its proof, given at the end of this section, will rely on  a decomposition principle 
for analytic functions on sectors,
going back to Franks and McIntosh \cite{FMI}, and from
ideas in the latter paper. 
In passing, this proof
provides a new one of \cite[Proposition 5.1]{KW1}.

In the next two lemmas, we fix two angles $0<\mu<\gamma<\pi$. 
The following can be extracted from \cite[Section 3]{FMI}.

\begin{lemma}\label{FM-1} 
There exist constants $\rho>1$ and $C>0$, and a family
$(\phi_{m,k,j})_{\substack{m=1,2; \\
k\in\footnotesize{\Zdb}; \,j\geq  0}}$
in $H^{\infty}_0(\Sigma_\mu)$ such that:
\begin{itemize}
\item [(1)] For any $n\in\Zdb$ and for any 
$z\in\Sigma_\mu$ satisfying
$\rho^n\leq\vert z\vert\leq \rho^{n+1}$, we have
$$
\bigl\vert \phi_{m,k,j}(z)\bigr\vert\,\leq\,
C\, 2^{-j}\,\rho^{-\frac12\vert k-n\vert},\qquad
m\in\{1,2\},\, k\in \Zdb,\, j\geq 0.
$$
\item [(2)] For any $f\in H^\infty(\Sigma_\gamma)$, there exists 
a family $(\alpha_{m,k,j})_{\substack{m=1,2; \\
k\in\footnotesize{\Zdb}; \, j\geq  0}}$
of complex numbers such that 
$$
\vert \alpha_{m,k,j}\vert\,\leq\,C\,\norm{f}_{\infty,\Sigma_\gamma}
,\qquad
m\in\{1,2\},\, k\in \Zdb,\, j\geq 0,
$$
and 
$$
f(z)=\,\sum_{\substack{m=1,2; \\
k\in\footnotesize{\Zdb};\, j\geq  0}}
\alpha_{m,k,j}\,\phi_{m,k,j}(z)\,,
\qquad z\in \Sigma_\mu.
$$
\item [(3)] For any $m\in\{1,2\},\, k\in \Zdb,\, j\geq 0,$ the function
$$
z\mapsto \phi_{m,k,j}(z)\Bigl(\frac{1+z^2}{z}\Bigr)^\frac12
$$
is bounded on $\Sigma_\mu$.
\end{itemize}
\end{lemma}

Let $\rho>1$ be given by Lemma \ref{FM-1}.
For any $k\in\Zdb$, we define $\sigma_k\in H^\infty(\Sigma_\mu)$ by 
\begin{equation}\label{sigmak}
\sigma_k(z)=\,\frac{\rho^{\frac{k}{4}}\, z^{\frac{1}{4}}}{
(\rho^ke^{i\gamma})^\frac12 \,-\, z^{\frac12}}\,,
\qquad z\in\Sigma_{\mu}.
\end{equation}

\begin{lemma}\label{FM-2}
There exist two constants $0<c_1<c_2$ such that for any
$k,n\in\Zdb$ and for any $z\in\Sigma_\mu$ satisfying
$\rho^n\leq\vert z\vert\leq \rho^{n+1}$, we have
$$
c_1\rho^{-\frac14\vert k-n\vert}\,\leq\vert\sigma_k(z)\vert\leq
c_2\rho^{-\frac14\vert k-n\vert}\,.
$$
\end{lemma}

\begin{proof}
Consider $z\in\Sigma_\mu$ satisfying
$\rho^n\leq\vert z\vert\leq \rho^{n+1}$.
On the one hand, we have
$$
\bigl\vert (\rho^ke^{i\gamma})^\frac12 \,-\, z^{\frac12}\bigr\vert
\leq \,\rho^{\frac{k}{2}}\, +\, \vert z\vert^\frac12
\,\leq \,(1+\rho^\frac12)\,
\rho^{\frac{1}{2}\max\{k,n\}}\,.
$$
On the other hand, 
we observe that for some $\delta>0$ (only depending 
on $\mu,\gamma$ and $\rho$), the distance
$\bigl\vert (\rho^ke^{i\gamma})^\frac12 \,-\, z^{\frac12}\bigr\vert\,$
is bounded from below by either 
$\delta\,\rho^{\frac{n}{2}}\,$ if $k\leq n$, or
$\delta\,\rho^{\frac{k}{2}}\,$ if $k\geq n$.
Thus we have
$$
\bigl\vert (\rho^ke^{i\gamma})^\frac12 \,-\, z^{\frac12}\bigr\vert
\geq \delta\,\rho^{\frac{1}{2}\max\{k,n\}}\,.
$$
These estimates imply 
$$
\vert\sigma_k(z)\vert \, \approx\, \frac{\rho^{\frac{k}{4}}\,
\vert z\vert^{\frac{1}{4}}}{\rho^{\frac{1}{2}\max\{k,n\}}}\,
\approx\,\rho^{\frac{k}{4}}\,\rho^{\frac{n}{4}}\,
\rho^{-\frac{1}{2}\max\{k,n\}}\,=\,\rho^{-\frac14\vert k-n\vert},
$$
under the condition $\rho^n\leq\vert z\vert\leq \rho^{n+1}$. This 
yields the result.
\end{proof}

The following decomposition of the unit (of independent interest)
will be derived from the above two lemmas.

\begin{proposition}\label{FM-3}
Let $\mu\in(0,\pi)$. There exist three sequences $(\Delta_i)_{i\geq 1}$,
$(\psi_i)_{i\geq 1}$ and $(\widetilde{\psi}_i)_{i\geq 1}$ 
in $H^\infty_0(\Sigma_\mu)$
satisfying the following properties.
\begin{itemize}
\item [(1)] There exists a constant $C\geq 0$ such that
$$
\forall\, z\in\Sigma_\mu,\qquad \sum_{i=1}^\infty \vert \psi_i(z)\vert\,\leq C
\qquad\hbox{and}\qquad
\sum_{i=1}^\infty \vert \widetilde{\psi}_i(z)\vert\,\leq C.
$$
\item [(2)] For any $\nu\in(0,\mu)$, 
there exists a constant $K\geq 0$ such that
$$
\forall\, i\geq 1,\qquad \int_{\partial\Sigma_\nu} 
\vert\Delta_i(z)\vert\, 
\Bigl\vert\frac{dz}{z}\Bigr\vert \ \leq K.
$$
\item [(3)] There exists a constant $C\geq 0$ such that
$$
\forall\, i\geq 1,\ \forall\, z\in\Sigma_\mu,\qquad 
\vert\Delta_i(z)\vert\,\leq C,
$$
and
$$
\forall\, z\in\Sigma_\mu,\qquad 1 = \sum_{i=1}^\infty \Delta_i(z)
\psi_i(z) \widetilde{\psi}_i(z).
$$
\end{itemize}
\end{proposition}

\begin{proof}
We fix $\gamma\in(\mu,\pi)$ and consider $\rho>1$
and the family $(\phi_{m,k,j})_{\substack{m=1,2; \\
k\in\footnotesize{\Zdb}; \,j\geq  0}}$ given by Lemma \ref{FM-1}.
We apply part (2) of the latter lemma to the constant function 
$f\equiv 1$. This yields a bounded family
$(\alpha_{m,k,j})_{\substack{m=1,2; \\
k\in\footnotesize{\Zdb}; \,j\geq  0}}$ of $\Cdb$ such that
\begin{equation}\label{1}
1 =\,\sum_{\substack{m=1,2; \\
k\in\footnotesize{\Zdb};\, j\geq  0}}
\alpha_{m,k,j}\,\phi_{m,k,j}(z)\,,
\qquad z\in \Sigma_\mu.
\end{equation}

Consider $m\in\{1,2\},\, k\in \Zdb,\, j\geq 0$.
We define
$$
\varphi_{m,k,j}\,=\,\frac{
\phi_{m,k,j}}{\sigma_k}\,.
$$
According to part (3) of Lemma \ref{FM-1}, this 
is an element of $H^\infty_0(\Sigma_\mu)$.
The inner-outer factorization of functions
in $H^\infty(\Sigma_\mu)$ provides
two functions $\psi_{m,k,j}$ and 
$\widetilde{\psi}_{m,k,j}$
in $H^\infty_0(\Sigma_\mu)$ such that 
$$
\varphi_{m,k,j}(z)\,=\,
\psi_{m,k,j}(z)\,
\widetilde{\psi}_{m,k,j}(z),\qquad z\in\Sigma_\mu,
$$
and 
$$
\vert \psi_{m,k,j}(z)
\vert \,=\, 
\vert \widetilde{\psi}_{m,k,j}(z)
\vert \,=\, \vert \varphi_{m,k,j}(z)
\vert^\frac12,\qquad \hbox{a.e.-}\,z\in\partial\Sigma_\mu.
$$

According to part (1) of Lemma \ref{FM-1} and 
the lower bound in Lemma \ref{FM-2}, 
there exists a constant $c_0>0$ such that
for any $z\in\Sigma_\mu$  satisfying
$\rho^n\leq\vert z\vert\leq \rho^{n+1}$, we have
$$
\bigl\vert \varphi_{m,k,j}(z)\bigr\vert\,\leq\,
c_0\, 2^{-j}\,\rho^{-\frac14\vert k-n\vert},\qquad
m\in\{1,2\},\, k\in \Zdb,\, j\geq 0.
$$
This implies that 
$$
\sup_{z\in\partial\Sigma_\mu} 
\sum_{\substack{m=1,2; \\
k\in\footnotesize{\Zdb};\, j\geq  0}}
\vert \psi_{m,k,j}(z)
\vert \,=\,\sup_{z\in\partial\Sigma_\mu} 
\sum_{\substack{m=1,2; \\
k\in\footnotesize{\Zdb};\, j\geq  0}}
\vert \varphi_{m,k,j}(z)
\vert^\frac12\,<\infty.
$$
Using the maximum principle of holomorphic functions we deduce that 
\begin{equation}\label{FM-5}
\sup_{z\in \Sigma_\mu} 
\sum_{\substack{m=1,2; \\
k\in\footnotesize{\Zdb};\, j\geq  0}}
\vert \psi_{m,k,j}(z)
\vert \,<\infty.
\end{equation}
Likewise we have
\begin{equation}\label{FM-6}
\sup_{z\in \Sigma_\mu} 
\sum_{\substack{m=1,2; \\
k\in\footnotesize{\Zdb};\, j\geq  0}}
\vert \widetilde{\psi}_{m,k,j}(z)
\vert \,<\infty.
\end{equation}

Let $\nu\in(0,\mu)$. Applying
the upper bound in Lemma \ref{FM-2}, we obtain
that for any $k\in\Zdb$,
we have
\begin{align*}
\int_{\partial\Sigma_\nu} 
\vert\sigma_k(z)\vert\, 
\Bigl\vert\frac{dz}{z}\Bigr\vert \
& =\,\sum_{n\in\footnotesize{\Zdb}}
\int_{\partial\Sigma_\nu\cap
\{\rho^n\leq\vert z\vert\leq \rho^{n+1}\}} 
\vert\sigma_k(z)\vert\, 
\Bigl\vert\frac{dz}{z}\Bigr\vert\\
&\leq c_2\,\sum_{n\in\footnotesize{\Zdb}}
\rho^{-\frac14 \vert k -n \vert}\,
\int_{\partial\Sigma_\nu\cap
\{\rho^n\leq\vert z\vert\leq \rho^{n+1}\}} 
\Bigl\vert\frac{dz}{z}\Bigr\vert.
\end{align*}
This last term is equal to 
$$
2c_2\,{\rm log}(\rho)\,\sum_{m\in
\footnotesize{\Zdb}}\rho^{-\frac{m}{4}}\,,
$$ 
which is independent
of $k$. Hence
\begin{equation}\label{FM-7}
\sup_{k\in\footnotesize{\Zdb}}
\int_{\partial\Sigma_\nu} 
\vert\sigma_k(z)\vert\, 
\Bigl\vert\frac{dz}{z}\Bigr\vert \ <\infty.
\end{equation}

Now define 
$$
\Delta_{m,k,j}\,=\,
\alpha_{m,k,j}\,\sigma_k,
$$
for any  $m\in\{1,2\},\, k\in \Zdb,\, j\geq 0$.
This is a uniformly bounded family  and 
(\ref{1}) now reads
$$
1 =\,\sum_{\substack{m=1,2; \\
k\in\footnotesize{\Zdb};\, j\geq  0}}
\Delta_{m,k,j}(z)\,\psi_{m,k,j}(z)\widetilde{\psi}_{m,k,j}(z)\,,
\qquad z\in \Sigma_\mu.
$$

Reindexing the families $(\Delta_{m,k,j})_{m,k,j}$, 
$(\psi_{m,k,j})_{m,k,j}$ and $(\widetilde{\psi}_{m,k,j})_{m,k,j}$
as sequences, the result follows from
this identity and the estimates (\ref{FM-5}), (\ref{FM-6})
and (\ref{FM-7}).
\end{proof}

\begin{proof}[Proof of Theorem \ref{d-KW}]
We assume that $(A_1,\ldots,A_d)$ admits an
$H^\infty(\Sigma_{\theta_1}\times\cdots
\times\Sigma_{\theta_d})$ joint functional
calculus. We consider $\theta'_k>\omega_k$, $k=1,\ldots,d$,
and we fix $\nu_k,\mu\in (0,\pi)$
such that
$$
\omega_k<\nu_k<\theta'_k<\mu \qquad \hbox{and}\qquad
\theta_k<\mu,\qquad k=1,\ldots,d.
$$
We consider the three sequences $(\Delta_i)_{i\geq 1}$,
$(\psi_i)_{i\geq 1}$ and $(\widetilde{\psi}_i)_{i\geq 1}$
from Proposition \ref{FM-3}. For any $i_1,\ldots,i_d\geq 1$, we 
let
$$
\Delta_{i_1,\ldots,i_d}=\Delta_{i_1}\otimes\cdots\otimes\Delta_{i_d}
\,\in H^{\infty}_0(\Sigma_{\mu}\times\cdots\times\Sigma_{\mu})
$$
denote the function taking any $(z_1,\ldots,z_d)$ in $\Sigma_{\mu}^d$
to 
$\Delta_{i_1}(z_1)\cdots\Delta_{i_d}(z_d)$. We similarly define
$\psi_{i_1,\ldots,i_d}=\psi_{i_1}\otimes\cdots\otimes\psi_{i_d}$
and $\widetilde{\psi}_{i_1,\ldots,i_d}=\widetilde{\psi}_{i_1}\otimes
\cdots\otimes\widetilde{\psi}_{i_d}$.
According to Proposition \ref{FM-3}, we have
\begin{equation}\label{FM-9}
\sup_{(z_1,\ldots,z_d)\in \Sigma_{\mu}^d}
\,\sum_{i_1,\ldots,i_d=1}^{\infty} \bigl\vert
\bigl(\Delta_{i_1,\ldots,i_d}\psi_{i_1,\ldots,i_d}
\widetilde{\psi}_{i_1,\ldots,i_d}\bigr)(z_1,\ldots,z_d)\bigr\vert\, <\infty 
\end{equation}
and
\begin{equation}\label{FM-8}
1=\,\sum_{i_1,\ldots,i_d=1}^{\infty} 
\Delta_{i_1,\ldots,i_d}\psi_{i_1,\ldots,i_d}
\widetilde{\psi}_{i_1,\ldots,i_d}
\end{equation}
on $\Sigma_{\mu}^d$.

Let $f\in H^\infty_0(\Sigma_{\theta'_1}\times\cdots\times\Sigma_{\theta'_d})$. 
For convenience we set $\A=(A_1,\ldots,A_d)$ and in this proof,
we write $f(\A)$ instead of $f(A_1,\ldots,A_d)$.  It is easy to deduce 
from (\ref{FM-8}), (\ref{FM-9}) and Fubini's theorem that 
$$
\sum_{i_1,\ldots,i_d=1}^{\infty}
\bignorm{f(\A)\Delta_{i_1,\ldots,i_d}(\A)\psi_{i_1,\ldots,i_d}(\A)
\widetilde{\psi}_{i_1,\ldots,i_d}(\A)}\,<\infty
$$
and 
$$
f(\A)= \,\sum_{i_1,\ldots,i_d=1}^{\infty}
f(\A)\Delta_{i_1,\ldots,i_d}(\A)\psi_{i_1,\ldots,i_d}(\A)
\widetilde{\psi}_{i_1,\ldots,i_d}(\A).
$$

We consider two arbitrary elements $x\in X$ and $y\in X^*$. 
For any integer $N\geq 1$, we set 
$$
f_N=\sum_{i_1,\ldots,i_d=1}^{N}
f \Delta_{i_1,\ldots,i_d} \psi_{i_1,\ldots,i_d}
\widetilde{\psi}_{i_1,\ldots,i_d}. 
$$
By Cauchy-Schwarz, we have 
\begin{align*}
\bigl\vert \bigl\langle
f_N(\A)x,y\bigr\rangle\bigr\vert\,&=
\Bigl\vert\sum_{i_1,\ldots,i_d=1}^{N}
\bigl\langle f(\A)\Delta_{i_1,\ldots,i_d}(\A)\psi_{i_1,\ldots,i_d}(\A)x,
\widetilde{\psi}_{i_1,\ldots,i_d}(\A)^*y\bigr\rangle\Bigr\vert\\
&\leq \Bignorm{
\sum_{i_1,\ldots,i_d=1}^{N}\varepsilon_{i_1,\ldots,i_d}\otimes
f(\A)\Delta_{i_1,\ldots,i_d}(\A)\psi_{i_1,\ldots,i_d}(\A)x}_{{\rm Rad}(X)}\\ 
& \times\,\Bignorm{
\sum_{i_1,\ldots,i_d=1}^{N}\varepsilon_{i_1,\ldots,i_d}\otimes
\widetilde{\psi}_{i_1,\ldots,i_d}(\A)^*y}_{{\rm Rad}(X^*)}.
\end{align*}
We shall now estimate each of the two factors in the upper bound.

According to the $R$-sectoriality assumption,
each set $\{zR(z,A_k)\,:\,z\in\partial\Sigma_{\nu_k}\}$ 
is $R$-bounded, hence by \cite[Proposition 8.1.19]{BOOK}, the product set 
$$
\biggl\{\prod_{k=1}^d \bigl(z_k R(z_k,A_k)\bigr)\,:\,
z_k\in\partial\Sigma_{\nu_k},\ k=1,\ldots,d\biggr\}
$$
is $R$-bounded.
Further it follows from part (2) of 
Proposition \ref{FM-3} that
$$
\int_{\prod_{k=1}^d \partial \Sigma_{\nu_k}}
\bigl\vert (f
\Delta_{i_1,\ldots,i_d})(z_1,\ldots,z_d)
\bigr\vert
\prod_{k=1}^d \Bigl\vert
\frac{dz_k}{z_k}
\Bigr\vert\ \lesssim \norm{f}_{\infty,\Sigma_{\theta'_1}\times
\cdots\times\Sigma_{\theta'_d}}.
$$
By (\ref{fAi}), we have
$$
f(\A)\Delta_{i_1,\ldots,i_d}(\A)
\,=\,
\Bigl(\frac{1}{2\pi i} \Bigr)^{d} 
\int_{\prod_{k=1}^d \partial \Sigma_{\nu_k}} (f
\Delta_{i_1,\ldots,i_d})(z_1,\ldots,z_d) 
\prod_{k=1}^d\bigl(z_k R(z_k,A_k)\bigr) \,
\prod_{k=1}^d \frac{dz_k}{z_k}.
$$
Applying \cite[Theorem 8.5.2]{BOOK}, we deduce from the above two 
results 
that the set 
$$
\F_f\,:=\bigl\{f(\A)\Delta_{i_1,\ldots,i_d}(\A)\, :\,
i_1,\ldots, i_d\geq 1\bigr\}
$$
is $R$-bounded, with an estimate 
$$
\R(\F_f)\lesssim 
\norm{f}_{\infty,\Sigma_{\theta'_1}\times
\cdots\times\Sigma_{\theta'_d}}.
$$
Consequently, 
\begin{align*}
\Bignorm{
\sum_{i_1,\ldots,i_d=1}^{N}\varepsilon_{i_1,\ldots,i_d}\otimes
f(\A)&\Delta_{i_1,\ldots,i_d}(\A)\psi_{i_1,\ldots,i_d}(\A)x}_{{\rm Rad}(X)}
\\ & \lesssim\,\norm{f}_{\infty,\Sigma_{\theta'_1}\times
\cdots\times\Sigma_{\theta'_d}}
\Bignorm{\sum_{i_1,\ldots,i_d=1}^{N}\varepsilon_{i_1,\ldots,i_d}\otimes
\psi_{i_1,\ldots,i_d}(\A)x}_{{\rm Rad}(X)}.
\end{align*}
By part (1) of Proposition \ref{FM-3}, 
there exists a constant $C\geq 0$ such that
$$
\Bignorm{\sum_{i_1,\ldots,i_d=1}^N
\eta_{i_1,\ldots,i_d}\psi_{i_1,\ldots,i_d}}_{\infty, \Sigma_{\mu}^d}\,\leq\,C
$$
for any $\eta_{i_1,\ldots,i_d}=\pm 1$. Since 
$(A_1,\ldots,A_d)$ has an
$H^\infty(\Sigma_{\theta_1}\times\cdots
\times\Sigma_{\theta_d})$ joint functional
calculus, this implies an estimate
$$
\Bignorm{\sum_{i_1,\ldots,i_d=1}^N
\eta_{i_1,\ldots,i_d}
\psi_{i_1,\ldots,i_d}(\A)}\,\leq\, C',
\qquad
\eta_{i_1,\ldots,i_d}=\pm 1.
$$
Consequently, we have an estimate 
$$
\Bignorm{\sum_{i_1,\ldots,i_d=1}^{N}\varepsilon_{i_1,\ldots,i_d}\otimes
\psi_{i_1,\ldots,i_d}(\A)x}_{{\rm Rad}(X)}
\,\lesssim\,\norm{x},
$$
and hence an estimate
$$
\Bignorm{
\sum_{i_1,\ldots,i_d=1}^{N}\varepsilon_{i_1,\ldots,i_d}\otimes
f(\A)\Delta_{i_1,\ldots,i_d}(\A)\psi_{i_1,\ldots,i_d}
(\A)x}_{{\rm Rad}(X)}\,\lesssim\,\norm{f}_{\infty,
\Sigma_{\theta'_1}\times
\cdots\times\Sigma_{\theta'_d}}\,\norm{x},
$$
not depending on $N$.

The second factor in the majorization
of $\vert\langle f_N(\A)x,y\rangle\vert$ can be treated similarly
and we obtain an estimate
$$
\Bignorm{
\sum_{i_1,\ldots,i_d=1}^{N}\varepsilon_{i_1,\ldots,i_d}\otimes
\widetilde{\psi}_{i_1,\ldots,i_d}(\A)^*y}_{{\rm Rad}(X^*)}
\lesssim\,\norm{y},
$$
not depending on $N$.

Altogether we thus have
$$
\bigl\vert \bigl\langle
f_N(\A)x,y\bigr\rangle\bigr\vert\,
\,\lesssim\,\norm{f}_{\infty,\Sigma_{\theta'_1}\times
\cdots\times\Sigma_{\theta'_d}}\,\norm{x}\,\norm{y}.
$$
Passing to the limit and taking the supremum over all $x,y$ of norms 
less than or equal to 1, we obtain 
$$
\norm{f(A_1,\ldots,A_d)}\,\lesssim\,\norm{f}_{\infty,\Sigma_{\theta'_1}\times
\cdots\times\Sigma_{\theta'_d}}.
$$
This  estimate can be proved as well for any
$f\in H_{0,1}^\infty(\Sigma_{\theta'_1}\times\cdots\times\Sigma{\theta'_d})$,
which yields the result.
\end{proof}

\section{Square functions for commuting families}

In this section we introduce square functions
associated with commuting families of sectorial operators
and establish connections with 
$H^\infty$ joint functional calculus.
As in the single case, the definition of square functions 
on general Banach spaces will
require the use of the Kalton-Weis $\gamma$-spaces so we first
supply a short introduction to these spaces.

Throughout we let $X$ be a Banach space and we let $H$ be a 
Hilbert space. We identify the algebraic tensor product 
$H\otimes X$ with the subspace of $B(H^*,X)$ of all 
continuous finite rank operators in the usual way. Let $(g_j)_{j\geq 1}$ 
be an independent sequence of complex valued standard Gaussian variables on
some probability space. For any $u\in H\otimes X$, there exists 
a finite orthonormal sequence $(e_j)_{j\geq 1}$ of $H$
and a finite sequence
$(x_j)_{j\geq 1}$ of $X$
such that $u=\sum_j e_j\otimes x_j$. According to (\ref{Rot}),
the quantity
$$
\norm{u}_\gamma\, :=\,\Bignorm{\sum_{j\geq 1} g_j\otimes x_j}_{G(X)}
$$
does not depend on the choice of the sequences 
$(e_j)_{j\geq 1}$ and 
$(x_j)_{j\geq 1}$ representing $u$. 
It turns out that $\norm{\,\cdotp}_\gamma$ is a norm on
$H\otimes X$. The Kalton-Weis space $\gamma(H^*;X)$
is defined as the completion of $H\otimes X$ with respect to this norm.
For any $u\in H\otimes X$, the operator norm of $u$
is smaller than or equal to $\norm{u}_\gamma$. Hence the inclusion map
$H\otimes X\subset B(H^*,X)$ extends to a contractive embedding
$$
\gamma(H^*;X)\,\subset B(H^*,X).
$$
These $\gamma$-spaces first appeared
in a preliminary version of \cite{KW2}. We refer to this paper 
and to \cite[Chapter 9]{BOOK} for some background. We will 
refer several times to the latter book for some
properties of the $\gamma$-spaces that we do not list here. 
We mention however two of these properties.
First we have the following description of $\gamma(H^*;X)$
as a subspace of $ B(H^*,X)$
(see \cite[Theorem 9.1.20]{BOOK}).

\begin{lemma}\label{c0}
For any $u\in\gamma(H^*;X)$, we have
\begin{equation}\label{g-norm}
\norm{u}_{\gamma(H^*;X)} = \,\sup\Bigl\{
\Bignorm{\sum_{j} g_j\otimes u(e_j^*)}_{G(X)}\Bigr\},
\end{equation}
where the supremum runs over all finite orthonormal 
families $(e_j^*)_{j}$
of $H^*$. 

Moreover
if $X$ does not contain $c_0$, then
an operator $u\in B(H^*,X)$ belongs to $\gamma(H^*;X)$
if and only if the supremum in the right-hand side of (\ref{g-norm})
is finite. 
\end{lemma}

Second, we have
the following fundamental tensor extension property,
which is a straighforward consequence of (\ref{Rot})
(see \cite[Theorem 9.1.10]{BOOK}).

\begin{lemma}\label{Tensorisation} For any $S\in B(H)$,
the mapping $S\otimes I_X\colon H\otimes X \to
H\otimes X$ (uniquely) extends to a bounded map
$$
S\overline{\otimes} I_X\colon \gamma(H^*;X)\longrightarrow
\gamma(H^*;X),
$$
with $\norm{S\overline{\otimes} I_X}=\norm{S}$.
\end{lemma}

We will consider the special case when $H=L^2(\Omega)$ for
some measure space $(\Omega,dm)$. We identify $L^2(\Omega)^*$
with $L^2(\Omega)$ through the standard duality pairing
$$
\langle h_1,h_2\rangle\,=\,\int_\Omega h_1(t)h_2(t)\, dm(t),\qquad
h_1,\, h_2\in L^2(\Omega,dm).
$$
Thus we have inclusions
$L^2(\Omega)\otimes X\subset 
\gamma(L^2(\Omega);X)\subset B(L^2(\Omega),X)$.

Let $\zeta\colon \Omega\to X$ be a measurable function. We say that
$\zeta$ is weakly-$L^2$ provided that
the function $y\circ\zeta$
belongs to $L^2(\Omega)$
for any $y\in X^*$. In this case, one can
define
a bounded operator $u_\zeta\colon L^2(\Omega)\to X\,$ by
$$
\langle y, u_\zeta(h)\rangle\, =\,
\int_{\Omega}  
\langle y, \zeta(t)\rangle  h(t)\,dm(t),\qquad h\in L^2(\Omega),\
 y\in X^*.
$$
The above formula a priori 
defines an operator $u_\zeta$
from $L^2(\Omega)$ into $X^{**}$; 
however it turns out that its range is included in $X$. 
We refer to \cite[Section 4]{KW2} and 
\cite[Section 9.2]{BOOK} for details.

We let $\gamma(\Omega;X)$ denote the space of 
all weakly-$L^2$ measurable functions
$\zeta\colon \Omega\to X$ such that $u_\zeta$ belongs
to $\gamma(L^2(\Omega);X)$. For such a function
$\zeta$, we set
$$
\norm{\zeta}_{\gamma(\Omega;X)}\, :=\, 
\norm{u_\zeta}_{\gamma(L^2(\Omega);X)}.
$$
It is plain to see that through the identification
$\zeta\longleftrightarrow u_\zeta$, 
$\gamma(\Omega;X)$ contains 
$L^2(\Omega)\otimes X$. Thus $\gamma(\Omega;X)$
is a dense subspace of $\gamma(L^2(\Omega);X)$.

\begin{lemma}\label{Bessel}
Assume that $X$ does not contain $c_0$ and let
$\zeta\colon \Omega\to X$ be a weakly-$L^2$ measurable function
such that $\zeta\cdotp h\in L^1(\Omega;X)$ for any $h\in L^2(\Omega)$.
Then $u_\zeta$ belongs to $\gamma(\Omega;X)$ if and only
if there exists a constant $C\geq 0$ such that
$$
\Bignorm{\sum_{j} g_j\otimes \Bigl(
\int_\Omega \zeta(t) e_j(t)\,dm(t)\Bigr)}_{G(X)}
\,\leq\,C
$$
for any  finite orthonormal sequence $(e_j)_{j}$
of $L^2(\Omega)$. In this case, $\norm{\zeta}_{\gamma(\Omega;X)}$
is the smallest constant $C$ satisfying this property.
\end{lemma}

\begin{proof}
It  is clear that for any $j$, we have
$\int_\Omega \zeta(t) e_j(t)\,dm(t)\,= u_\zeta(e_j)$.
Hence the
result follows from Lemma \ref{c0}.
\end{proof}

In the sequel we will work with the measure space 
$\Omega_0 = (\Rdb_+^*,\frac{dt}{t})$
and with its powers $\Omega_0^d$, equipped with
\begin{equation}\label{mes}
dM(t) : = \,\frac{dt_1}{t_1}\,\cdots\frac{dt_d}{t_d}\,,\qquad 
t=(t_1,\ldots,t_d)\in\Omega_0^d.
\end{equation}

\bigskip
Throughout the rest of this section,
we let $(A_1,\ldots,A_d)$ be a commuting family of sectorial
operators on $X$, with respective types $\omega_1,\ldots,\omega_d$.
Let $F\in H^\infty_{0,1}(\Sigma_{\nu_1}\times\cdots\times\Sigma_{\nu_d})$, 
with $\nu_k\in(\omega_k,\pi)$ for all 
$k=1,\ldots,d$. By Lebesgue's Theorem and (\ref{fAi}), the function 
$$
t=(t_1,\ldots,t_d)\,\mapsto\, F(t_1A_1,\ldots,t_dA_d)
$$
is continuous
from $\Omega_0^d$ into $B(X)$. Let $x\in X$. 
If $t\mapsto F(t_1A_1,\ldots,t_dA_d)x$
belongs to $\gamma(\Omega_0^d;X)$, then we set
$$
\norm{x}_F = \bignorm{t\mapsto F(t_1A_1,\ldots,t_dA_d)x}_{\gamma(\Omega_0^d;X)}.
$$
We set $\norm{x}_F=\infty$ otherwise.

Extending the usual terminology for single 
sectorial operators (see \cite{KW2, HH, BOOK}), 
we introduce the following.

\begin{definition}
We say that $(A_1,\ldots,A_d)$ admits a square 
function estimate with respect to $F\in H^\infty_{0,1}(\Sigma_{\nu_1}
\times\cdots\times\Sigma_{\nu_d})$
if there exists a constant
$K\geq 0$ such that 
\begin{equation}\label{SFE}
\norm{x}_F\,\leq K\,\norm{x},\qquad x\in X.
\end{equation}
\end{definition}

The main result of this section is the following theorem,
which generalizes \cite[Theorem 10.4.16]{BOOK}
(see also \cite[Proposition 7.7]{KW2},
\cite[Corollary 6.7]{CDMY} and \cite[Theorem 7.6]{JLX})
from the single to the multivariable case.

\begin{theorem}\label{H-SFE}
Assume that $X$ has finite cotype and that $(A_1,\ldots,A_d)$ admits an
$H^\infty(\Sigma_{\theta_1}\times\cdots\times\Sigma_{\theta_d})$ joint functional
calculus. 
Then for any $F\in H^\infty_{0,1}(\Sigma_{\nu_1}\times\cdots\times\Sigma_{\nu_d})$, with
$\nu_k\in(\theta_k,\pi)$, $(A_1,\ldots,A_d)$ admits a square function estimate with respect to $F$.
\end{theorem}

The proof of this theorem will use the following proposition, which
extends \cite[Theorem 6.3]{LM}. In the case
of a single sectorial operator, a similar
approach was developed in \cite{HH} (see also \cite[10.4.c]{BOOK}).

\begin{proposition}\label{Quad}
Assume that $X$ has finite cotype and that $(A_1,\ldots,A_d)$ admits an
$H^\infty(\Sigma_{\theta_1}\times\cdots\times\Sigma_{\theta_d})$ 
joint functional calculus. Then for any $\nu_k\in(\theta_k,\pi)$, $k=1,\ldots,d$,
there exists a constant $K\geq 0$ such that 
$$
\Bignorm{\sum_{j=1}^n r_j\otimes F_j(A_1,\ldots,A_d)x}_{{\rm Rad}(X)}
\,\leq\, K\norm{x}\,\Bignorm{\Bigl(\sum_{j=1}^n\vert F_j\vert^2\Bigr)^{\frac12}}_{\infty,
\Sigma_{\nu_1}\times\cdots\times\Sigma_{\nu_d}},
$$
for any $n\geq 1$, for any $F_1,\ldots,F_n$
in $H^\infty_0(\Sigma_{\nu_1}\times\cdots\times\Sigma_{\nu_d})$ 
and for any $x\in X$.
\end{proposition}

The proof of  \cite[Theorem 6.3]{LM} (which corresponds to
Proposition \ref{Quad} 
in the case $d=1$) relies on the 
Franks-McIntosh decomposition stated as \cite[Theorem H.3.1]{BOOK}.
The latter is a simple consequence of Lemma \ref{FM-1}
(and its proof).
In the multivariable
case, the following version of the Franks-McIntosh decomposition
holds true: for any $0<\theta_k<\nu_k<\pi$, $k=1,\ldots,d$,
there exist sequences $(\varphi_{k,i})_{i\geq 1}$
and $(\psi_{k,i})_{i\geq 1}$ in $H_0^\infty(\Sigma_{\theta_k})$, as well as 
a constant $C\geq 0$, such that:
\begin{itemize}
\item [(a)] For any $k=1,\ldots,d$, and for any $z\in \Sigma_{\theta_k}$,
$$
\sum_{i=1}^\infty\vert\varphi_{k,i}(z)\vert\,\leq\,C
\qquad\hbox{and}\qquad 
\sum_{i=1}^\infty\vert\psi_{k,i}(z)\vert\,\leq\,C;
$$
\item [(b)] For any Banach space $Y$ and for any $F\in H^\infty(
\Sigma_{\nu_1}\times\cdots\times\Sigma_{\nu_d};Y)$, there exists
a bounded family $(y_{i_1,\ldots,i_d})_{i_1,\ldots,i_d\geq 1}$
in $Y$ such that 
$$
\norm{y_{i_1,\ldots,i_d}}\leq C\,\norm{F}_{\infty, 
\Sigma_{\nu_1}\times\cdots\times\Sigma_{\nu_d}},\qquad i_1,\ldots,i_d\geq 1,
$$
and for any $(z_1,\ldots,z_d)\in \Sigma_{\theta_1}
\times\cdots\times\Sigma_{\theta_d}$,
$$
F(z_1,\ldots,z_d)\,=\,\sum_{i_1,\ldots, i_d =1}^{\infty} 
y_{i_1,\ldots,i_d}\,\varphi_{1,i_1}(z_1)\psi_{1,i_1}(z_1)\cdots
\varphi_{d,i_d}(z_d)\psi_{d,i_d}(z_d).
$$
\end{itemize}
This well-known extension of \cite[Theorem H.3.1]{BOOK}
is implicit in \cite[Section 4]{FMI}.

Using this decomposition, the proof of Proposition \ref{Quad}
is a repetition of the proof of \cite[Theorem 6.3]{LM}, so we skip it.
In the latter, the finite cotype assumption allows to apply
\cite[Corollary 3.4]{KW}, which plays a crucial role.

\begin{proof}[Proof of Theorem \ref{H-SFE}]
We may and do assume that 
$F\in H^\infty_{0}(\Sigma_{\nu_1}\times\cdots\times\Sigma_{\nu_d})$.
If $x\in N(A_k)$ for some $k\in\{1,\ldots,d\}$, then 
$F(t_1A_1,\ldots, t_d A_d)x=0$ 
for any $t=(t_1,\ldots,t_d)\in\Omega_0^d$,
hence $\norm{x}_F=0$. It therefore follows from
Lemma \ref{Ergodic2} that to prove an estimate (\ref{SFE}),
we may assume that
\begin{equation}\label{x}
x\in \overline{R(A_1)}\cap\cdots\cap \overline{R(A_d)}.
\end{equation}

Recall (\ref{mes}).
Let $(e_j)_{j\geq 1}$ be an orthonormal sequence
of $L^2(\Omega_0^d)$. 
For any
$j\geq 1$, define a function 
$F_j\colon \Sigma_{\nu_1}\times \cdots\times 
\Sigma_{\nu_d}\to\Cdb$ by
$$
F_j(z_1,\ldots,z_d)= \int_{\Omega_0^d}
F(t_1 z_1,\ldots,t_d z_d)\,e_j(t)\,
dM(t)\,.
$$
It is clear that $F_j\in H^\infty(\Sigma_{\nu_1}
\times\cdots\times
\Sigma_{\nu_d})$. (Note that $F_j$ does not necessarily belong
to $H^\infty_{0}(\Sigma_{\nu_1}
\times\cdots\times
\Sigma_{\nu_d})$; 
this is why we need the approximation process
below.)

For any integer $m\geq 1$, consider $\Phi_m$
and $\Phi_m^{\otimes d}$
defined by (\ref{Phi_m}) 
and (\ref{Phi_m^d}). Then
$F_j\Phi_m^{\otimes d}$ belongs to $H^\infty_{0}(\Sigma_{\nu_1}
\times\cdots\times
\Sigma_{\nu_d})$ for any $j\geq 1$.

The argument in the proof of \cite[Lemma 6.5]{JLX} shows that 
the function 
$$
t\mapsto\, F(t_1 A_1,\ldots,t_d A_d)\Phi_m(A_1)\cdots\Phi_m(A_d)
$$
belongs to $L^2(\Omega_0^d;B(X))$. Then for any
$j\geq 1$,
the function 
$$
t\mapsto F(t_1 A_1,\ldots,t_d A_d)\Phi_m(A_1)\cdots\Phi_m(A_d) e_j(t)
$$ 
belongs to $L^1(\Omega_0^d;B(X))$ and 
by Fubini's theorem,
we obtain that
\begin{equation}\label{IntRep}
(F_j\Phi_m^{\otimes d})(A_1,\ldots,A_d)\,=\,
\int_{\Omega_0^d}
F(t_1 A_1,\ldots,t_d A_d)\Phi_m(A_1)\cdots\Phi_m(A_d)
\, e_j(t)\,
dM(t)\,.
\end{equation}
According to Proposition \ref{Quad} and to the finite cotype
assumption, there is a constant $K$
such that
\begin{equation}\label{n}
\Bignorm{\sum_{j=1}^n g_j\otimes (F_j\Phi_m^{\otimes d})
(A_1,\ldots,A_d)x}_{{\rm G}(X)}
\,\leq\, K\norm{x}\,\Bignorm{\Bigl(\sum_{j=1}^n\vert F_j\Phi_m^{\otimes d}\vert^2\Bigr)^{\frac12}}_{\infty,
\Sigma_{\nu_1}\times\cdots\times\Sigma_{\nu_d}},
\end{equation}
for any $n\geq 1$, $m\geq 1$ and $x\in X$.
Let $(z_1,\ldots,z_d)\in \Sigma_{\nu_1}\times\cdots\times\Sigma_{\nu_d}$. 
Since $(e_j)_{j\geq 1}$ is an orthonormal family of
$L^2(\Omega^d_0)$, we have
$$
\Bigl(\sum_{j=1}^n \vert F_j(z_1,\ldots,z_d)
\vert^2\Bigr)^\frac12
\,\leq \Bigl(\int_{\Omega_0^d}
\vert F(t_1 z_1,\ldots,t_d z_d)\vert^2\,
dM(t)\,\Bigr)^\frac12.
$$
By assumption, $F\in H^\infty_0(\Sigma_{\nu_1}\times
\cdots\times\Sigma_{\nu_d})$ hence by (\ref{Hinfini0}) 
there exist real numbers $C\geq 0$ and
$s_1,\ldots,s_d >0$ (not depending on $z_1,\ldots,z_d)$
such that
\begin{align*}
\Bigl(\sum_{j=1}^n \vert F_j(z_1,\ldots,z_d)\vert^2\Bigr)^\frac12
\,& \leq C \Bigl(\int_{\Omega_0^d}
\prod_{k=1}^d \dfrac{|t_k z_k|^{s_k}}{(1+|t_kz_k|^{s_k})^2}\,
dM(t)\,\Bigr)^\frac12\\
& = C\prod_{k=1}^d
\Bigl(\int_{0}^\infty
\dfrac{|t_k z_k|^{s_k}}{(1+|t_kz_k|^{s_k})^2}\,
\frac{dt_k}{t_k}\,\Bigr)^\frac12.
\end{align*}
A standard calculation shows
that the right-hand side 
of the above inequality
is bounded on $\Sigma_{\nu_1}\times\cdots\times\Sigma_{\nu_d}$.
Moreover the sequence $(\phi^{\otimes d}_m)_{m\geq 1}$ is 
uniformly bounded
on $\Sigma_{\nu_1}\times\cdots\times\Sigma_{\nu_d}$. 
Implementing these 
estimates in (\ref{n}) we obtain a constant $K'>0$ such that
$$
\Bignorm{\sum_{j=1}^n g_j\otimes (F_j\Phi_m^{\otimes d})(A_1,\ldots,A_d)x}_{{\rm G}(X)}
\,\leq\, K'\norm{x},\qquad x\in X,
$$
for all $n,m\geq 1$.

Since $X$ has finite cotype, it does not contain $c_0$.
Hence according to (\ref{IntRep}) and Lemma \ref{Bessel}, the
above estimate shows that 
the function 
$$
(t_1,\ldots,t_d)\mapsto F(t_1A_1,\ldots,t_d A_d)
\Phi_m(A_1)\cdots\Phi_m(A_d)x
$$ 
belongs to $\gamma(\Omega_0^d;X)$,
and that its norm in the space
$\gamma(L^2(\Omega_0^d);X)$  is $\leq K'\norm{x}$.

Now assume that $x$ satifies (\ref{x}). Applying (\ref{Approx})
to $A_k$ for any $k=1,\ldots,d$, we obtain that 
$\Phi_m(A_1)\cdots\Phi_m(A_d)x\to x$ when $m\to\infty$.
It therefore follows from \cite[Lemma 4.10 (b)]{KW2} that 
the function 
$(t_1,\ldots,t_d)\mapsto F(t_1A_1,\ldots,t_d A_d)
x$ belongs to $\gamma(\Omega_0^d;X)$
and that its norm in the space
$\gamma(L^2(\Omega_0^d);X)$ is $\leq K'\norm{x}$.
This proves the square function estimate.
\end{proof}

It is well-known that if $X$ is reflexive, then
each $A_k^*$ is a sectorial operator
on $X^*$. It is plain to see that if $(A_1,\ldots,A_d)$ admits an
$H^\infty(\Sigma_{\theta_1}\times\cdots\times\Sigma_{\theta_d})$ joint functional
calculus, then $(A_1^*,\ldots,A_d^*)$ admits an
$H^\infty(\Sigma_{\theta_1}\times\cdots\times\Sigma_{\theta_d})$ joint functional
calculus as well. Indeed for any $f\in 
H^\infty_{0,1}(\Sigma_{\theta_1}\times\cdots\times\Sigma_{\theta_d})$,
\begin{equation}\label{adjoint}
f(A_1^*,\ldots,A_d^*)\,=\,\widetilde{f}(A_1,\ldots,A_d)^*,
\end{equation}
where $\widetilde{f}\in 
H^\infty_{0,1}(\Sigma_{\theta_1}\times\cdots\times\Sigma_{\theta_d})$ is defined 
by $\widetilde{f}(z_1,\ldots,z_d)=\overline{f(\overline{z_1},
\ldots,\overline{z_d})}$.

It therefore follows from Theorem \ref{H-SFE} that when $X$ is reflexive
and $X$ and $X^*$ both have finite cotype, we have the following
result:
if
$(A_1,\ldots,A_d)$ admits an
$H^\infty(\Sigma_{\theta_1}\times\cdots\times\Sigma_{\theta_d})$ joint functional
calculus, then both $(A_1,\ldots,A_d)$ and $(A_1^*,\ldots,A_d^*)$
admit square function estimates with respect to any 
$F\in H^\infty_{0,1}(\Sigma_{\nu_1}\times\cdots\times\Sigma_{\nu_d})$, with
$\nu_k\in(\theta_k,\pi)$.

In the single case, Theorem \ref{H-SFE} 
has various converse statements. These
statements all assert that under certain assumptions, some
square function estimates for $A$ and $A^*$ imply
that $A$ has an $H^\infty$-functional calculus.
The next statement, involving $R$-sectoriality,
is essentially a multivariable version 
of \cite[Theorem 10.4.9]{BOOK}.
Theorem \ref{SFE-H} will not be used in the next section
and is given for the sake of completeness.

\begin{theorem}\label{SFE-H}
Assume that $X$ is reflexive with finite
cotype, and  that for any $k=1,\ldots,d$,
$A_k$ is $R$-sectorial of $R$-type $\omega_k\in(0,\pi)$
and has dense range. Assume further that there  
exist non zero
functions $F_1,F_2\in 
H^\infty_{0}(\Sigma_{\nu_1}\times\cdots\times\Sigma_{\nu_d})$, 
for some $\nu_k\in(\omega_k,\pi)$, 
such that $(A_1,\ldots,A_d)$ admits a square
function estimate with respect to $F_1$ and 
$(A_1^*,\ldots, A_d^*)$ admits a square 
function estimate with respect to $F_2$. 
Then for any $\theta_k\in(\omega_k,\pi)$,
$k=1,\ldots,d$, $(A_1,\ldots,A_d)$ admits an
$H^\infty(\Sigma_{\theta_1}\times\cdots\times\Sigma_{\theta_d})$ 
joint functional
calculus. 
\end{theorem}

\begin{proof} We use classical arguments so we will be deliberately brief.
As in the proof of Theorem \ref{d-KW}, we set  $\A=(A_1,\ldots,A_d)$
and we write $F(t\A)$ instead of $F(t_1A_1,\ldots,t_dA_d)$
for any $t=(t_1,\ldots,t_d)\in\Rdb_+^{* d}$, whenever $F
\in H_0^\infty(\Sigma_{\theta_1}\times\cdots\times\Sigma_{\theta_d})$, with
$\theta_k\in(\omega_k,\pi)$.

We may assume that $\theta_k\leq\nu_k$ for any $k=1,\ldots,d$. 
Define 
$\widetilde{F_2}\in
H^\infty_{0}(\Sigma_{\nu_1}\times\cdots\times\Sigma_{\nu_d})$
so that
\begin{equation}\label{adjoint2}
\widetilde{F_2}(A_1,\ldots,A_d)^*\,=\, F_2(A_1^*,\ldots,A_d^*),
\end{equation}
see (\ref{adjoint}).
Since $F_1,F_2$ are non zero, we can find $\Psi\in
H^\infty_{0}(\Sigma_{\nu_1}\times\cdots\times\Sigma_{\nu_d})$ such that 
$$
\int_{\Omega_0^d} \Psi(t)F_1(t)\widetilde{F_2}(t)\, dM(t)\, =1.
$$
By analytic continuation, this implies that for any $(z_1,\ldots,z_d)
\in \Sigma_{\theta_1}\times\cdots\times\Sigma_{\theta_d}$, we have
$$
\int \Psi(t_1z_1,\ldots,t_dz_d)F_1(t_1z_1,\ldots,t_dz_d)
\widetilde{F_2}(t_1z_1,\ldots,t_dz_d)\, dM(t)\, =1.
$$
Fix 
$f\in H^\infty_0(\Sigma_{\theta_1}\times\cdots
\times\Sigma_{\theta_d})$. The argument in the proof of 
\cite[Lemma 6.5 (1)]{JLX} shows that 
the function
$t\mapsto f(\A)\Psi(t\A)F_1(t\A)\widetilde{F_2}(t\A)$ 
belongs to $L^1(\Omega_0^d; B(X))$ and by Fubini's theorem,
\begin{equation}\label{Ident}
\int_{\Omega_0^d} 
f(\A)\Psi(t\A)F_1(t\A)\widetilde{F_2}(t\A)\, dM(t)\, =f(\A).
\end{equation}
Arguing as in the proof of  Theorem \ref{d-KW}, we may deduce 
from the assumption that $A_k$ is $R$-sectorial
of $R$-type $\omega_k$ that the set
$$
\F_f\, :=\,\bigl\{f(\A)\Psi(t\A)\, :\, t \in 
\Rdb_+^{* d}\bigr\}
$$
is $R$-bounded, with an estimate 
\begin{equation}\label{RFf}
\R(\F_f)\lesssim \norm{f}_{\infty,
\Sigma_{\theta_1}\times\cdots\times\Sigma_{\theta_d}}.
\end{equation}

We assumed that 
$(A_1,\ldots,A_d)$ admits a square
function estimate with respect to $F_1$. Thus for any
$x\in X$, $t\mapsto F_1(t\A)x$ belongs to $\gamma(\Omega_0^d;X)$, 
with an estimate
$$
\norm{x}_{F_1}
= \norm{t\mapsto F_1(t\A)x}_{\gamma(\Omega_0^d;X)}\,\lesssim\,\norm{x}.
$$
Since $X$ has finite cotype, it therefore
follows from (\ref{RFf}) and the so-called
$\gamma$-multiplier theorem (see \cite[Theorem 9.5.1]{BOOK})
that the function $t\mapsto f(\A)\Psi(t\A)F_1(t\A)x$ 
belongs to $\gamma(\Omega_0^d;X)$
for any $x\in X$, with an estimate
\begin{equation}\label{Est1}
\norm{t\mapsto f(\A)\Psi(t\A)F_1(t\A)x}_{\gamma(\Omega_0^d;X)}\,
\lesssim\,\norm{f}_{\infty,
\Sigma_{\theta_1}\times\cdots\times\Sigma_{\theta_d}}\norm{x}.
\end{equation}
Furthermore $(A_1^*,\ldots,A_d^*)$ admits a square
function estimate with respect to $F_2$ hence for any
$y\in X^*$, the function
$t\mapsto F_2(t\A^*)y$ belongs to $\gamma(\Omega_0^d;X^*)$, with an estimate
\begin{equation}\label{Est2}
\norm{y}_{F_2} = 
\norm{t\mapsto F_2(t\A^*)y}_{\gamma(\Omega_0^d;X^*)}\,\lesssim\,\norm{y}.
\end{equation}
Let us now use the identities (\ref{Ident}) and (\ref{adjoint2}). 
Applying the so-called trace duality theorem on $\gamma$-spaces (see 
\cite[Theorem 9.2.14]{BOOK}), the estimates (\ref{Est1}) and
(\ref{Est2}) imply that 
\begin{align*}
\bigl\vert\langle  f(\A)x,y\rangle\bigr\vert\,
& =\, \Bigl\vert
\int_{\Omega_0^d} 
\bigl\langle f(\A)\Psi(t\A)F_1(t\A)x, F_2(t\A^*)y\bigr\rangle\, dM(t)\Bigr\vert\\
& \leq \,
\bignorm{t\mapsto \Psi(t\A)F_1(t\A)x}_{\gamma(\Omega_0^d;X)}
\, \bignorm{t\mapsto F_2(t\A)^*y}_{\gamma(\Omega_0^d;X^*)}\\
&\lesssim 
\norm{f}_{\infty,
\Sigma_{\theta_1}\times\cdots\times\Sigma_{\theta_d}}\norm{x}\norm{y},
\end{align*}
for $x\in X, y\in X^*$. This yields an estimate
$$
\norm{f(\A)}\,\lesssim\, \norm{f}_{\infty,
\Sigma_{\theta_1}\times\cdots\times\Sigma_{\theta_d}}
$$
for $f\in H^\infty_0(\Sigma_{\theta_1}\times\cdots
\times\Sigma_{\theta_d})$. By Lemma \ref{DenseRange}
and the dense range assumption, this shows
that 
$(A_1,\ldots,A_d)$ admits an
$H^\infty(\Sigma_{\theta_1}\times\cdots\times\Sigma_{\theta_d})$ 
joint functional
calculus. 
\end{proof}

\begin{remark}\label{Converse} 
\ 

\smallskip
(1) 
The following variant of Theorem \ref{SFE-H} (without any dense range
assumption) is easy to
deduce from the above proof: Assume that $X$ is reflexive with finite
cotype, and  that for any $k=1,\ldots,d$,
$A_k$ is $R$-sectorial of $R$-type $\omega_k\in(0,\pi)$.
Let $\nu_k\in (\omega_k,\pi)$ for any $k=1,\ldots,d$, and
assume that for any $F\in H_{0,1}^\infty(
\Sigma_{\nu_1}\times\cdots\times\Sigma_{\nu_d})$, 
$(A_1,\ldots,A_d)$ and $(A_1^*,\ldots, A_d^*)$ both admit a square
function estimate with respect to $F$. 
Then for any $\theta_k\in(\omega_k,\pi)$,
$k=1,\ldots,d$, $(A_1,\ldots,A_d)$ admits an
$H^\infty(\Sigma_{\theta_1}\times\cdots\times\Sigma_{\theta_d})$ 
joint functional calculus. 

\smallskip
(2) Combining the above part (1) and Theorem \ref{H-SFE},
we obtain a new proof of Theorem \ref{d-KW} in the case when 
$X$ is reflexive and both $X$ and $X^*$ have finite cotype.
\end{remark}

We refer the reader to \cite[p. 325]{KW1} for the definition of
property $(\Delta)$. The latter is called `triangular contraction
property' in \cite{BOOK}, see Definition 7.5.7 in the latter reference.
We recall that any space with $(\Delta)$ has finite cotype.
It is proved in \cite[Theorem 5.3]{KW1} that if $A$ is a sectorial operator 
with an $H^\infty(\Sigma_\theta)$ functional calculus on
$X$ with property $(\Delta)$, then $A$ is $R$-sectorial 
of $R$-type $\theta$. Combining that result with 
Theorem \ref{H-SFE} and Remark \ref{Converse} (1), we 
obtain the following.

\begin{corollary} Assume that $X$ is a reflexive Banach space with 
property $(\Delta)$ and that $X^*$ has finite cotype. 
The following are equivalent.
\begin{itemize}
\item [(i)] For any $\theta_k\in(\omega_k,\pi)$, $k=1,\ldots,d$,
$(A_1,\ldots,A_d)$ admits an $H^\infty(\Sigma_{\theta_1}\times\cdots\times\Sigma_{\theta_d})$ 
joint functional calculus.
\item [(ii)] Each $A_k$ is $R$-sectorial of $R$-type $\omega_k$
and for any $\theta_k\in(\omega_k,\pi)$, $k=1,\ldots,d$,
and any $F\in H_{0,1}^\infty(\Sigma_{\theta_1}
\times\cdots\times\Sigma_{\theta_d})$,
$(A_1,\ldots,A_d)$ and $(A_1^*,\ldots, A_d^*)$ both admit a square
function estimate with respect to $F$. 
\end{itemize}
\end{corollary}

In the rest of this section, we discuss the following natural
problem : 

\begin{itemize}
\item [{\bf (P)}]
{\it Let $F_k\in H_0^\infty(\Sigma_{\nu_k})$, with
$\nu_k\in(\omega_k,\pi)$ for all $k=1,\ldots,d$. Consider
$F_1\otimes\cdots\otimes F_d$ in $H^\infty_0(\Sigma_{\nu_1}\times
\cdots\times\Sigma_{\nu_d})$ given by
$(F_1\otimes\cdots\otimes F_d)(z_1,\ldots, z_d)= F_1(z_1)\cdots
F_d(z_d)$. Assume that for each $k=1,\ldots,d$, 
$A_k$ admits a square function estimate with respect to $F_k$.
Does it imply that the $d$-tuple $(A_1,\ldots,A_d)$  admits a square
function estimate with respect to $F_1\otimes\cdots\otimes F_d$?
}
\end{itemize}

\smallskip
Assume first that $d=2$. Observe that if $A_1$ (resp.
$A_2$) admits a square function estimate with respect to $F_1$
(resp. $F_2$), then for any $x\in X$, the function
$(t_1,t_2)\mapsto F_1(t_1A_1)F_2(t_2A_2)x$ belongs to
$\gamma(\Omega_0;\gamma(L^2(\Omega_0);X))$, with an estimate
\begin{equation}\label{JointSFE}
\bignorm{(t_1,t_2)\mapsto 
F_1(t_1A_1)F_2(t_2A_2)x}_{\gamma(\Omega_0;\gamma(L^2(\Omega_0;X))}\,\lesssim\,\norm{x}.
\end{equation}
Indeed for $i=1,2$,
let $\Gamma_i\colon X\to \gamma(L^2(\Omega_0);X)$ be the operator
taking any $x\in X$ to the function $t\mapsto F_i(tA_i)x$.
Let $\widetilde{\Gamma_2}\colon \gamma(L^2(\Omega_0);X)
\to \gamma(L^2(\Omega_0); \gamma(L^2(\Omega_0);X))$ be the tensor
extension of $\Gamma_2$ defined by $\widetilde{\Gamma_2}(u)=\Gamma_2\circ u$
for any $u\in\gamma(L^2(\Omega);X)$. Then the operator
$$
\widetilde{\Gamma_2}\circ\Gamma_1\colon X\longrightarrow
\gamma(L^2(\Omega_0); \gamma(L^2(\Omega_0);X))
$$
takes any $x\in X$
to the function 
$(t_1,t_2)\mapsto F_1(t_1A_1)F_2(t_2A_2)x$. 
This yields the estimate
(\ref{JointSFE}).

It follows from this observation that  $(A_1,A_2)$ 
admits a square
function estimate with respect to $F_1\otimes F_2$ provided that 
the natural inclusion $L^2(\Omega_0)\otimes L^2(\Omega_0)\otimes X\to
L^2(\Omega_0^2)\otimes X$ extends to a bounded map
\begin{equation}\label{Type2}
\gamma(L^2(\Omega_0);\gamma(L^2(\Omega_0);X))\,\longrightarrow
\gamma(L^2(\Omega_0^2);X).
\end{equation}
By a simple reiteration argument, the above result extends to any $d\geq 2$. 
That is, if $X$ satisfies (\ref{Type2}), if
$A_k$ admits a square function estimate with respect to $F_k$
for any $k=1,\ldots,d$, then
$(A_1,\ldots,A_d)$  admits a square
function estimate with respect to $F_1\otimes\cdots\otimes F_d$.

We note that property (\ref{Type2}) is satisfied if 
either $X$ has property $(\alpha)$
or $X$ has type 2 \cite[Proposition 2.7]{VVW}. Consequently,
the above problem ${\bf (P)}$ has an affirmative answer if
either $X$ has property $(\alpha)$
or $X$ has type 2.

\begin{example}\label{Ex}
We now turn to a counterexample. For any 
$1<p<\infty$, let $S^p$ denote the $p$-Schatten class
on $\ell^2$. Let $(e_k)_{k\geq 1}$ denote the
standard basis of $\ell^2$ and let
$c\in B(\ell^2)$ be the positive selfadjoint operator satisfying
$c(e_k) = 2^{-k}e_k$ for any $k\geq 1$. Let $A_1, A_2\colon S^p\to S^p$
be the left and right
multiplication operators defined by $A_1(x)=cx$
and $A_2(x)=xc$, respectively, for
$x\in S^p$. It is well-known that $A_1,A_2$ are sectorial operators
with dense range, and that 
they admit an $H^\infty(\Sigma_\theta)$ functional calculus for
any $\theta\in(0,\pi)$ (see e.g. \cite[Section 8.A]{JLX} for details).
Further $A_1$ and $A_2$ commute. It follows from the proof 
of \cite[Theorem 3.9]{LLLM} that if $p\not=2$, 
then for any $\theta_1,\theta_2\in (0,\pi)$,
the couple $(A_1,A_2)$ does not have an 
$H^\infty(\Sigma_{\theta_1}\times\Sigma_{\theta_2})$ 
joint functional calculus.

Assume that $1<p<2$ and let $p'=\frac{p}{p-1}$ be its conjugate number.
Let $F_1,F_2$ be non zero functions in $H^\infty_0(\Sigma_{\nu})$ for 
some $\nu\in(0,\pi)$, and let $F=F_1\otimes F_2$. Since 
$A_1$ admits an $H^\infty(\Sigma_\theta)$ functional calculus 
for any $\theta\in(0,\pi)$, both $A_1$ and $A_1^*$ admit square function
estimates with respect to $F_1$. Likewise $A_2$ and $A_2^*$ admit square function
estimates with respect to $F_2$. 
The Banach space
$S^{p*}=S^{p'}$ has type 2, hence it follows from the discussion  
preceding this example that $(A_1^*,A_2^*)$ admits a 
square function
estimates with respect to $F$. If $(A_1,A_2)$ admitted a 
square function
estimates with respect to $F$ as well, it would follow from
Theorem \ref{SFE-H} that $(A_1,A_2)$ admits an 
$H^\infty(\Sigma_{\theta}\times\Sigma_{\theta})$ joint functional calculus
for some $\theta$, which is false. Consequently,
$(A_1,A_2)$ does not admit a
square function
estimate with respect to $F=F_1\otimes F_2$.
\end{example}

We refer the reader to \cite[Chapters 8-10]{JLX} for more examples
of sectorial operators on non-commutative $L^p$-spaces with
an $H^\infty$ functional calculus.

\section{Bounded $H^\infty$-functional calculus and dilation
on $K$-convex spaces}

This section is devoted to dilation properties of 
commuting families of bounded analytic
semigroups and their connections
with joint functional calculus. We start with some background on $K$-convexity and
state two elementary lemmas.

Let $(g_j)_{j\geq 1}$ be a independent family of complex valued
standard Gaussian variables on some probability space 
$(\S,\mathbb{P})$. Let $G\subset L^2(\S)$ be the closed linear span
of the $g_j$ and let $q\colon L^2(\S)\to L^2(\S)$ be the
orthogonal projection with range equal to $G$.

Let $X$ be a Banach space.
We say that $X$ is $K$-convex if $q\otimes I_X\colon 
L^2(\S)\otimes X\to L^2(\S)\otimes X$ extends to a bounded
map $q_X$ from $L^2(\S; X)$ into itself. In this case, the range
of $q_X$ is equal to $G(X)$, 
the closure of all finite sums $\sum_j g_j\otimes x_j$,
with $x_j\in X$. We refer the
reader to \cite{M} for various characterizations of 
$K$-convex spaces and more information. 
For readers not familiar with this notion we
mention that for any $1<p<\infty$, either classical 
or noncommutative $L^p$-spaces are $K$-convex.
Also we note that if $X$ is $K$-convex, then its dual space
$X^*$ is $K$-convex as well. Further any 
$K$-convex Banach space has finite cotype.

It follows from
the definition of $\gamma$-spaces that
if $H$ is any separable Hilbert space, then
$\gamma(H^*;X)$ is isometrically isomorphic
to $G(X)$.
In this section we will work
with $H=L^2(\Rdb^d)$, that we identify with its dual 
space in the usual way. It follows from
above that if $X$ is $K$-convex, we have a direct 
sum decomposition
\begin{equation}\label{K}
L^2(\S;X) = \gamma(L^2(\Rdb^d);X)\oplus Z_X,
\end{equation}
for some subspace $Z_X$ of $L^2(\S;X)$.

Assume that $X$ is a reflexive Banach space
and let 
$(A_1,\ldots,A_d)$ be a commuting $d$-tuple of
sectorial operators on $X$. For any 
$\Lambda\subset\{1,\ldots,d\}$, let $X_\Lambda$
be defined by (\ref{X-Lambda}). Recall
that $(A_1^*,\ldots,A_d^*)$ is a commuting $d$-tuple of
sectorial operators on $X^*$. Then we similarly
define $X^*_\Lambda\subset X^*$
as the intersection of $\bigcap_{k\in \Lambda} \overline{R(A_k^*)}$
and $\bigcap_{k\notin \Lambda} N(A_k^*)$. 
Let us use the notation $(X_\Lambda)^*$ for the dual
space of $X_\Lambda$. The closeness of the notations
$(X_\Lambda)^*$ and $X^*_\Lambda$ is justified by 
the fact they are isomorphic, as follows.

\begin{lemma}\label{Ergodic3}
The mapping $X^*_\Lambda\to (X_\Lambda)^*$ which 
takes any $\eta\in X^*_\Lambda$ to its restriction
$\eta_{\vert X_\Lambda}$, is an isomorphism.
\end{lemma}

\begin{proof}
For any $k=1,\ldots,d$, $N(A_k)^\perp =\overline{R(A_k^*)}$
and $\overline{R(A_k)}^\perp
=N(A_k^*)$. Hence we have
$$
X_\Lambda^\perp = \,
\bigoplus_{\substack{\Lambda'\subset\{1,\ldots,d\}\\ \Lambda'\not=\Lambda}} X^*_{\Lambda'}.
$$
Then the result follows from standard duality principles.
\end{proof}

Recall the sequence $(\Phi_m)_{m\geq 1}$ given by 
(\ref{Phi_m}). According to \cite[Lemma 6.5]{JLX},
which holds true on any Banach space, we have the
following fact.

\begin{lemma}\label{Integral}
Let $(T_t)_{t\geq 0}$ be a bounded analytic semigroup
on $X$, and let $A$ be its
negative generator. For any $m\geq 1$,
the function
$s\mapsto AT_s\Phi_m(A)$ belongs to $L^1(\Rdb_+^*;B(X))$, and
$$
\int_0^\infty A T_s\Phi_m(A)\,ds\, = \,\Phi_m(A).
$$
\end{lemma}

Let 
$\bigl((T^1_t)_{t\geq 0},\ldots, (T^d_t)_{t\geq 0}\bigr)$
be a $d$-tuple of bounded $C_0$-semigroups on 
$X$ and for any $k=1,\ldots,d$, let $A_k$ be the negative generator
of $(T^k_t)_{t\geq 0}$. It is plain to see that the sectorial operators
$A_1,\ldots,A_d$ commute
if and only if the semigroups $(T^1_t)_{t\geq 0},\ldots, (T^d_t)_{t\geq 0}$
commute, that is, $T^{k}_{t}T^{k'}_{t'}=T^{k'}_{t'}T^{k}_{t}$ for any 
$1\leq k,k'\leq d$ and any $t,t'\geq 0$.

The following theorem 
should be regarded as a multivariable version of 
\cite[Theorem 5.1]{FW} on $K$-convex reflexive spaces. 
It is also related albeit 
different from
\cite[Theorem 4.8]{ArLM}. Note that
Definition \ref{Bdd-FC2} is used in the statement.

\begin{theorem}\label{d-FW1}
Assume that $X$ is reflexive and $K$-convex. Let 
$\bigl((T^1_t)_{t\geq 0},\ldots, (T^d_t)_{t\geq 0}\bigr)$
be a $d$-tuple of commuting bounded analytic semigroups, and 
let $A_1,\ldots, A_d$ denote their negative generators.
Assume that $(A_1,\ldots,A_d)$ admits an
$H^\infty(\Sigma_{\theta_1}\times\cdots\times\Sigma_{\theta_d})$ joint functional
calculus, for some $\theta_k\in(0,\frac{\pi}{2})$, 
$k=1,\ldots,d$. Then there
exist a measure space $(\Omega, dm)$, two bounded operators
$J\colon X\to L^2(\Omega;X)$ and $Q\colon L^2(\Omega;X)\to X$,
as well as a $d$-tuple 
$\bigl((U^1_t)_{t\in\footnotesize{\Rdb}},\ldots, 
(U^d_t)_{t\in\footnotesize{\Rdb}}\bigr)$
of commuting bounded $C_0$-groups on $L^2(\Omega;X)$ such that:
\begin{itemize}
\item [(a)] The $d$-tuple $(B_1,\ldots,B_d)$, where $B_k$ is the negative 
generator of $(U^k_t)_{t}$ for all $k=1,\ldots,d$, admits 
an $H^\infty(\Sigma_{\frac{\pi}{2}}\times\cdots\times \Sigma_{\frac{\pi}{2}})$
joint functional calculus;
\item [(b)] For any $t_1,\ldots,t_d$ in $\Rdb_+$, we have
$$
T^1_{t_1}\cdots T^d_{t_d}\,=\, Q U^1_{t_1}\cdots U^d_{t_d} J.
$$
\end{itemize}
\end{theorem}

\begin{proof}
We set $X_1=\overline{R(A_1)}\cap\cdots\cap\overline{R(A_d)}$
and $X^*_1 = \overline{R(A_1^*)}\cap\cdots\cap\overline{R(A_d^*)}$.

For any $x\in X_1$, let $\varphi_x\colon\Rdb^d\to X$ be defined
by $\varphi_x(s_1,\ldots,s_d) = 
A_1^\frac12 T^1_{s_1}\cdots A_d^\frac12 T^d_{s_d}x$ if
$(s_1,\ldots,s_d)\in \Rdb_+^{*d}$, and $\varphi_x(s_1,\ldots,s_d) = 0$
otherwise.

Consider $F\in 
H_0^\infty(\Sigma_{\theta_1}\times\cdots\times\Sigma_{\theta_d})$
defined by 
\begin{equation}\label{F}
F(z_1,\ldots,z_d) =\,\bigl(z_1^\frac12 e^{-z_1}\bigr)
\bigl(z_2^\frac12 e^{-z_2}\bigr)\cdots
\bigl(z_d^\frac12 e^{-z_d}\bigr).
\end{equation}
By the joint functional calculus assumption and
Theorem \ref{H-SFE}, $(A_1,\ldots,A_d)$ admits a
square function estimate with respect to $F$. Since
$$
F(s_1A_1,\ldots,s_d A_d)=s_1^\frac12
A_1^\frac12 T^1_{s_1}\cdots s_d^\frac12A_d^\frac12 T^d_{s_d}
$$
for any $s_1,\ldots,s_d >0$,
this implies that
$\varphi_x$ belongs to $\gamma(\Rdb^d;X)$ for any $x$ and that 
we have a uniform estimate
$$
\norm{x}_F=
\norm{\varphi_x}_{\gamma(\Rdb^d;X)}\,\lesssim\,
\norm{x},\qquad x\in X_1.
$$
Using the embedding of $\gamma(L^2(\Rdb^d);X)$
into $L^2(\S;X)$ discussed at the beginning
of this section, this yields a bounded operator
$$
J_1\colon X_1\longrightarrow L^2(\S;X)
$$
given by $J_1(x) = \varphi_x$, for any $x\in X_1$.

Likewise we have a bounded operator
$$
\widetilde{J}_1\colon X^*_1\longrightarrow L^2(\S;X^*)
$$
given by $\widetilde{J}_1(y) = \widetilde{\varphi}_y$ for any 
$y\in X^*_1$, where $\widetilde{\varphi}_y(s_1,\ldots,s_d) = 
A_1^{* \frac12} T^{1 *}_{s_1}\cdots A_d^{* \frac12}  T^{d *}_{s_d} y$ if
$(s_1,\ldots,s_d)\in \Rdb_+^{*d}$, and $\widetilde{\varphi}_y
(s_1,\ldots,s_d) = 0$
otherwise.

For any $k=1,\ldots,d$,
and any $t\in\Rdb$,
let $\tau^k_t\colon L^2(\Rdb^d)\to L^2(\Rdb^d)$
be the translation operator defined by setting
\begin{equation}\label{Tau}
[\tau^k_t(h)](s_1,\ldots,s_d)
= h(s_1,\ldots,s_{k-1},s_k+t,s_{k+1},\ldots,s_d),
\qquad (s_1,\ldots,s_d)\in\Rdb^d,
\end{equation}
for any $h\in L^2(\Rdb^d)$.
According to Lemma \ref{Tensorisation}, 
$\tau^k_t\otimes I_X$ extends to an 
isometric isomorphism 
$$
V^k_t\colon\gamma(L^2(\Rdb^d);X)
\longrightarrow \gamma(L^2(\Rdb^d);X).
$$
A thorough look at this tensor extension process shows that
$V_t^k(u)=u\circ\tau^k_{-t}$ for any $u\in\gamma(L^2(\Rdb^d);X)$.
We deduce that for any $\zeta\in\gamma(\Rdb^d;X)$, we have
\begin{equation}\label{Action1}
[V^k_t(u_\zeta)]=u_{\zeta_t^k},
\end{equation}
where $\zeta_t^k\in\gamma(\Rdb^d;X)$ is given by
\begin{equation}\label{Action2}
\zeta_t^k(s_1,\ldots,s_d)
=\zeta(s_1,\ldots,s_{k-1},s_k+t,s_{k+1},\ldots,s_d),
\qquad (s_1,\ldots,s_d)\in\Rdb^d.
\end{equation}
Moreover for any $k=1,\ldots,d$,
$(V_t^k)_{t\in\footnotesize{\Rdb}}$ is an
isometric $C_0$-group on $\gamma(L^2(\Rdb^d);X)$.
Using (\ref{K}), which follows from the $K$-convexity of $X$, we now define 
$$
U^k_t\colon L^2(\S;X)\longrightarrow
L^2(\S;X)
$$
by setting $U^k_t(u,w)= (V^k_t(u),w)$
for any $u\in\gamma(L^2(\Rdb^d);X)$ and any $w\in Z_X$.
It follows from above that 
$\bigl((U^1_t)_{t\in\footnotesize{\Rdb}},\ldots, 
(U^d_t)_{t\in\footnotesize{\Rdb}}\bigr)$
is a commuting family of
bounded $C_0$-groups on $L^2(\S;X)$.

Let $x\in X_1$, let $y\in X_1^*$
and let $t_1,\ldots, t_d\geq 0$. Since
$X$ is reflexive, we have
\begin{equation}\label{L2X}
L^2(\S;X^*)=L^2(\S;X)^*
\end{equation}
isometrically, see e.g. \cite[Section IV.1]{DU}.
Then according to (\ref{Action1}), (\ref{Action2}) 
and trace duality on $\gamma$-spaces
(see \cite[Section 9.1.f]{BOOK}), we have
$$
\bigl\langle
U^1_{t_1}\cdots U^d_{t_d} \,J_1(x),
\widetilde{J}_1(y)\bigr\rangle\,
=\, \int_{\Rdb^d}\bigl\langle
\varphi_x(s_1+t_1,\ldots, s_d+t_d),
\widetilde{\varphi}_y(s_1,\ldots,s_d)
\bigr\rangle\, ds_1\cdots ds_d\,.
$$

Let $x_m=\Phi_m(A_1)\cdots \Phi_m(A_d)x$ for any $m\geq 1$.
Using the definitions of $\varphi_x, \widetilde{\varphi}_y$
and the above identity, we obtain that 
\begin{align*}
\bigl\langle
U^1_{t_1}\cdots U^d_{t_d} &
J_1(x_m),\widetilde{J}_1(y)\bigr\rangle\,
 \\ &
=\,\int_0^\infty\cdots\int_0^\infty
\bigl\langle (A_1 T^1_{2s_1+t_1}\Phi_m(A_1))\cdots
(A_d T^d_{2s_d+t_d}\Phi_m(A_1))(x),y\bigr\rangle\, ds_1\cdots ds_d
\\ & 
=\,\int_0^\infty\cdots\int_0^\infty
\bigl\langle (A_1 T^1_{2s_1}\Phi_m(A_1))\cdots
(A_d T^d_{2s_d}\Phi_m(A_1))(T^1_{t_1}\cdots T^d_{t_d}
x),y\bigr\rangle\, ds_1\cdots ds_d\,.
\end{align*}
Applying Lemma \ref{Integral} 
for each $A_k$, we deduce that
$$
\bigl\langle
U^1_{t_1}\cdots U^d_{t_d} J_1(x_m),\widetilde{J}_1 (y)\bigr\rangle\,
=\,2^{-d}\bigl\langle
T^1_{t_1}\cdots T^d_{t_d}x_m,y\bigr\rangle.
$$ 
Since $x\in X_1$, 
we know that $x_m\to x$, by (\ref{Approx}). 
We infer that
$$
\bigl\langle
U^1_{t_1}\cdots U^d_{t_d} 
J_1(x),\widetilde{J}_1(y)\bigr\rangle\,
=\,2^{-d}\bigl\langle
T^1_{t_1}\cdots T^d_{t_d}x,y\bigr\rangle.
$$
Using Lemma \ref{Ergodic3}, let us identify $X_1^*$ with $(X_1)^*$
and note that 
$L^2(\S;X^*)^*= L^2(\S;X)$, by (\ref{L2X}). Set 
$Q_1=2^d\widetilde{J}_1^*\colon L^2(\S;X)\to
X_1$. Then the above identity shows that 
$$
Q_1 U^1_{t_1} \cdots U^d_{t_d} J_1 = T^1_{t_1}\cdots
{T^d_{t_d}}_{\big\vert X_1}.
$$

For any $k=1,\ldots, d$, let 
$B_k$ (resp. $C_k$) be the negative generator of the $C_0$-group
$(U_t^k)_{t}$ (resp. $(V_t^k)_{t}$). Then let $c_k$
be the negative generator of the $C_0$-group $(\tau^k_t)_t$ on $L^2(\Rdb^d)$.
It is easy to check that for any $\theta_1,\ldots,\theta_d$
in $(\frac{\pi}{2},\pi)$ and any $f\in H^\infty_{0,1}(\Sigma_{\theta_1}\times
\cdots\times\Sigma_{\theta_d})$, the operator
$f(C_1,\ldots,C_d)\colon \gamma(L^2(\Rdb^d);X)\to  
\gamma(L^2(\Rdb^d);X)$ coincides with
$f(c_1,\ldots,c_d)\otimes I_X$ on 
$L^2(\Rdb^d)\otimes X$. According to
Lemma \ref{Tensorisation}, this implies
that 
$$
\norm{f(C_1,\ldots,C_d)}_{\gamma(L^2;X)\to \gamma(L^2;X)}
=\norm{f(c_1,\ldots,c_d)}_{L^2\to L^2}.
$$
Now since 
$(\tau^1_t)_t,\ldots,(\tau^d_t)_t$ are 
commuting unitary groups, it follows from spectral theory that 
$\norm{f(c_1,\ldots,c_d)}\leq\norm{f}_{\infty,\Sigma_{\frac{\pi}{2}}^d}$. 
This proves that $(C_1,\ldots,C_d)$
admits 
an $H^\infty(\Sigma_{\frac{\pi}{2}}\times\cdots\times \Sigma_{\frac{\pi}{2}})$
joint functional calculus. It readily implies
that the same holds true for $(B_1,\ldots,B_d)$.

What we have proved so far is that the theorem holds true 
on $X_1$ instead of $X$. Now it is easy to adapt the argument 
to show that more generally, the theorem holds true 
on $X_\Lambda$ for any $\Lambda\subset\{1,\ldots,d\}$, using
$$
F_\Lambda(z_1,\ldots,z_d) = \,\prod_{k\in\Lambda}\bigl(z_k^{\frac12}
e^{-z_k}\bigr)
$$
instead of (\ref{F}). 

Finally
the result holds true on $X$ using Lemma \ref{Ergodic2}
and a simple direct sum argument.
\end{proof}

\begin{remark}
Using the $\gamma^p$-spaces (see \cite[Section 9.1]{BOOK}) 
in place of the $\gamma$-spaces considered in the present paper,
we easily obtain that for any $1<p<\infty$,
Theorem \ref{d-FW1} holds true as well with $L^p(\Omega;X)$
in place of $L^2(\Omega;X)$.
\end{remark}

The following  result provides a converse to Theorem \ref{d-FW1},
up to an $R$-sectoriality assumption on the $A_k$.

\begin{theorem}\label{d-FW2}
Let 
$\bigl((T^1_t)_{t\geq 0},\ldots, (T^d_t)_{t\geq 0}\bigr)$
be a $d$-tuple of commuting bounded analytic semigroups on
some Banach space $X$, and let $A_1,\ldots, A_d$ denote their negative generators.
Assume that each $A_k$ is $R$-sectorial of $R$-type $\omega_k<\frac{\pi}{2}$.
Assume further
that there
exist a Banach space $Y$, two bounded operators
$J\colon X\to Y$ and $Q\colon Y\to X$,
as well as a $d$-tuple $\bigl((U^1_t)_{t\in\footnotesize{\Rdb}},\ldots, 
(U^d_t)_{t\in\footnotesize{\Rdb}}\bigr)$
of commuting bounded $C_0$-groups on $Y$ such that:
\begin{itemize}
\item [(a)] The $d$-tuple $(B_1,\cdots,B_d)$, where $B_k$ is the negative 
generator of $(U^k_t)_{t}$ for all $k=1,\ldots,d$, admits 
an $H^\infty(\Sigma_{\theta_1}\times\cdots\times \Sigma_{\theta_d})$
joint functional calculus, for some $\theta_k\in [\frac{\pi}{2},\pi)$;
\item [(b)] For any $t_1,\ldots,t_d$ in $\Rdb_+$, we have
$$
T^1_{t_1}\cdots T^d_{t_d}\,=\, Q U^1_{t_1}\cdots U^d_{t_d}J.
$$
\end{itemize}
Then $(A_1,\ldots,A_d)$ admits an
$H^\infty(\Sigma_{\theta'_1}\times\cdots\times\Sigma_{\theta'_d})$ joint functional
calculus, for all $\theta'_k\in(\omega_k,\frac{\pi}{2})$, $k=1,\ldots,d$.
\end{theorem}

\begin{proof} By condition (b) and the Laplace
formula, we have
$$
R(z_1,A_1)\cdots R(z_d,A_d) = QR(z_1,B_1)\cdots R(z_d,B_d)J.
$$
for any complex numbers  $z_1,\cdots, z_d$ 
with negative real parts.  This implies, by (\ref{fAi}), that
$$
f(A_1,\ldots,A_d)\,=\, Qf(B_1,\ldots,B_d)J,
$$
and hence
\begin{equation}\label{domination}
\norm{f(A_1,\ldots,A_d)}\leq\,\norm{Q}\norm{J}\,
\norm{f(B_1,\ldots,B_d)},
\end{equation}
for any $\theta_1,\ldots,\theta_d \in (\frac{\pi}{2},\pi)$ and any
$f\in H_{0,1}^\infty
(\Sigma_{\theta_1}\times\cdots\times \Sigma_{\theta_d})$.

By condition (a), we may choose $\theta_1,\ldots,\theta_d$ in such a way that 
$(B_1,\cdots,B_d)$ admits 
an $H^\infty(\Sigma_{\theta_1}\times\cdots\times \Sigma_{\theta_d})$
joint functional calculus. Then by (\ref{domination}),
$(A_1,\ldots,A_d)$ admits an
$H^\infty(\Sigma_{\theta_1}\times\cdots\times\Sigma_{\theta_d})$ joint functional
calculus as well. Now applying Theorem \ref{d-KW}, we obtain
the result.
\end{proof}

Recall 
that if $X$ has property $(\Delta)$, then any sectorial operator
$A$ on $X$
with an $H^\infty(\Sigma_\theta)$ functional calculus 
is automatically $R$-sectorial 
of $R$-type $\theta$ \cite[Theorem 5.3]{KW1}. 
Combining this result with Theorems \ref{d-FW1} and \ref{d-FW2}, we 
obtain the following corollary. 
We note that any UMD Banach
space is both reflexive, $K$-convex (see e.g. \cite[Section 4.3]{BOOK0})
and has property $(\Delta)$ (see \cite[Proposition 3.2]{KW1});
hence the next result applies to such spaces.

\begin{corollary}
Let $X$ be a reflexive, $K$-convex Banach space, 
with property $(\Delta)$. Let 
$\bigl((T^1_t)_{t\geq 0},\ldots, (T^d_t)_{t\geq 0}\bigr)$
be a $d$-tuple of commuting bounded analytic semigroups 
on $X$, and 
let $A_1,\ldots, A_d$ denote their negative generators.
The following assertions are equivalent.
\begin{itemize}
\item [(i)] 
The $d$-tuple
$(A_1,\ldots,A_d)$ admits an
$H^\infty(\Sigma_{\theta_1}\times\cdots\times\Sigma_{\theta_d})$ joint functional
calculus, for some $\theta_k\in(0,\frac{\pi}{2})$, 
$k=1,\ldots,d$. 
\item [(ii)] Each $A_k$ is $R$-sectorial of $R$-type $<\frac{\pi}{2}$
and there exist a measure space $(\Omega, dm)$, two bounded operators
$J\colon X\to L^2(\Omega;X)$ and $Q\colon L^2(\Omega;X)\to X$,
as well as a $d$-tuple 
$\bigl((U^1_t)_{t\in\footnotesize{\Rdb}},\ldots, 
(U^d_t)_{t\in\footnotesize{\Rdb}}\bigr)$
of commuting bounded $C_0$-groups on $L^2(\Omega;X)$ such that:
\begin{itemize}
\item [(a)] The $d$-tuple $(B_1,\ldots,B_d)$, where $B_k$ is the negative 
generator of $(U^k_t)_{t}$ for all $k=1,\ldots,d$, admits 
an $H^\infty(\Sigma_{\frac{\pi}{2}}\times\cdots\times \Sigma_{\frac{\pi}{2}})$
joint functional calculus;
\item [(b)] For any $t_1,\ldots,t_d$ in $\Rdb_+$, we have
$$
T^1_{t_1}\cdots T^d_{t_d}\,=\, Q U^1_{t_1}\cdots U^d_{t_d} J.
$$
\end{itemize}
\end{itemize}
\end{corollary}

To conclude, we wish to emphasize that for
a commuting family of generators of bounded $C_0$-groups,
to have an 
$H^\infty(\Sigma_{\frac{\pi}{2}}\times
\cdots\times \Sigma_{\frac{\pi}{2}})$
joint functional calculus is a very strong property.

\begin{proposition}\label{HP0}
Let $\bigl((U^1_t)_{t\in\footnotesize{\Rdb}},\ldots, 
(U^d_t)_{t\in\footnotesize{\Rdb}}\bigr)$ be
a commuting family of bounded $C_0$-groups on some Banach
space $Y$. The following are equivalent.
\begin{itemize}
\item [(i)] The $d$-tuple $(B_1,\ldots,B_d)$, where $B_k$ is the negative 
generator of $(U^k_t)_{t}$ for all $k=1,\ldots,d$, admits 
an $H^\infty(\Sigma_{\frac{\pi}{2}}\times
\cdots\times \Sigma_{\frac{\pi}{2}})$
joint functional calculus.
\item [(ii)] There exists a constant $K\geq 0$ such that for any
$b\in L^1(\Rdb^d)$,
\begin{equation}\label{HP}
\Bignorm{\int_{\Rdb^d} b(t_1,\ldots,t_d)U^1_{t_1}\cdots
U^d_{t_d}\, dt_1\cdots dt_d}\leq\, K\norm{\widehat{b}}_{\infty,\Rdb^d}.
\end{equation}
\end{itemize}
\end{proposition}

In this statement, $\widehat{b}\colon\Rdb^d\to\Cdb$
denotes the Fourier transform of 
$b\in L^1(\Rdb^d)$ and the integral in the left-hand side 
of (\ref{HP}) is defined in the strong sense. Property
(ii) means that the mapping 
$b\mapsto \int_{\Rdb^d} b(t_1,\ldots,t_d)U^1_{t_1}\cdots
U^d_{t_d}\, dt_1\cdots dt_d$ ``extends" to a bounded homomorphism 
$C_0(\Rdb^d)\to B(Y)$. More precisely, there exists
a (necessarily unique) bounded homomorphism
$\rho\colon C_0(\Rdb^d)\to B(Y)$ such that
$$
\rho(\widehat{b}) = \,\int_{\Rdb^d} b(t_1,\ldots,t_d)U^1_{t_1}\cdots
U^d_{t_d}\, dt_1\cdots dt_d,\qquad b\in L^1(\Rdb^d).
$$
This property is optimal since by Fourier-Plancherel,
$\norm{\widehat{b}}_{\infty,\Rdb^d}$ coincides with the norm of the operator
$\int_{\Rdb^d} b(t_1,\ldots,t_d)U^1_{t_1}\cdots
U^d_{t_d}\, dt_1\cdots dt_d\,$ in the case when $Y=L^2(\Rdb^d)$
and $U_t^k= \tau_t^k$ is given by (\ref{Tau}).
See \cite{LM2} for more on this theme.

\begin{proof}[Proof of Proposition \ref{HP0}] 
For any non empty $\Lambda\subset \{1,\ldots,d\}$ and 
any $\beta \in L^1(\Rdb_+^\Lambda)$,
we let $L_\beta\colon\overline{\Sigma_{\frac{\pi}{2}}^d}\to\Cdb\,$ 
be the Laplace transform of $\beta$ defined by
$$
L_\beta(z_1,\ldots,z_d)=\,\int_{\Rdb_+^\Lambda} \beta((t_k)_{k\in\Lambda})\,
\prod_{k\in\Lambda} e^{-z_kt_k}\,\prod_{k\in\Lambda} dt_k,
\qquad (z_1,\ldots,z_d)\in \overline{\Sigma_{\frac{\pi}{2}}^d}.
$$
This function only depends on the variables $(z_k)_{k\in\Lambda}$.

\smallskip
\underline{(ii)$\,\Rightarrow\,$(i):} 
We fix $\theta\in\bigl(\frac{\pi}{2},\pi\bigr)$.
Let $f\in H^{\infty}_0(\Sigma_\theta^\Lambda)$, for some
$\Lambda\subset \{1,\ldots,d\}$, $\Lambda\not=\emptyset$.
Arguing as in \cite[Section 3.3]{Haa}, we find that 
there exists $\beta \in L^1(\Rdb_+^\Lambda)$
such that $L_\beta$ coincides with $f$ on $\Sigma_{\frac{\pi}{2}}^d$
and 
\begin{equation}\label{Lap}
f(B_1,\ldots,B_d) =\int_{\Rdb^\Lambda} \beta((t_k)_{k\in\Lambda})
\prod_{k\in\Lambda} U^k_{t_k}\, \prod_{k\in\Lambda} dt_k\,.
\end{equation}
For any integer $n\geq 1$, define
$$
b_n(t_1,\ldots,t_d) = \beta((t_k)_{k\in\Lambda})\,\prod_{j\notin\Lambda}
(ne^{-nt_j}),\qquad
t_1,\ldots,t_d\in\Rdb_+.
$$
Then $b_n\in L^1(\Rdb_+^d)$ and 
$$
\widehat{b_n}(s_1,\ldots,s_d) = \widehat{\beta}((s_k)_{k\in\Lambda})
\prod_{j\notin\Lambda}\Bigl(\frac{n}{n+is_j}\Bigr),
\qquad s_1,\ldots,s_d\in\Rdb.
$$
Consequently, 
$$
\norm{\widehat{b_n}}_{\infty,\Rdb^d} = 
\norm{f}_{\infty,\Sigma_{\frac{\pi}{2}}^d},
\qquad n\geq 1.
$$
Moreover for any $n\geq 1$, we have
$$
\int_{\Rdb^d} b_n(t_1,\ldots,t_d)U^1_{t_1}\cdots
U^d_{t_d}\, dt_1\cdots dt_d \,
=\,(-1)^{d-\vert\Lambda\vert}
f(B_1,\ldots,B_d)
\,\prod_{j\notin\Lambda}\bigl(n(nI_Y +B_j)^{-1}\bigr)
$$
by (\ref{Lap}) and the Laplace formula. Hence (ii) implies 
that for any $n\geq 1$,
$$
\Bignorm{f(B_1,\ldots,B_d)
\,\prod_{j\notin\Lambda}\bigl(n(nI_Y +B_j)^{-1}\bigr)}
\leq K \norm{f}_{\infty,\Sigma_{\frac{\pi}{2}}^d}
$$
Since $n(nI _Y+B_j)^{-1}\to I_Y$ strongly, the
above inequality yields
$\norm{f(B_1,\ldots,B_d)}
\leq K \norm{f}_{\infty,\Sigma_{\frac{\pi}{2}}^d}$.
According to the comment after \cite[Definition 2.2]{ArLM},
this implies (i).

\smallskip
\underline{(i)$\,\Rightarrow\,$(ii):} Assume (i)
and set 
$K_k=\sup\{\norm{U^k_t}\, :\, t\in\Rdb\}$ for any $k=1,\ldots,d$. 
To prove (ii), it
clearly suffices to prove an estimate 
(\ref{HP}) for $b$ lying in a dense subset of $L^1(\Rdb^d)$.
Let $b\in L^1([-s,\infty))\otimes\cdots \otimes L^1([-s,\infty))$
for some $s>0$.
Let $b_s\in L^1(\Rdb^d)$ be defined by $b_s(t_1,\ldots,t_d)
=b(t_1-s,\ldots,t_d -s)$ for any $(t_1,\ldots, t_d)\in\Rdb_d$.
Then 
\begin{align*}
\int_{\Rdb^d} b(t_1,\ldots,t_d)U^1_{t_1}\cdots
U^d_{t_d}\, dt_1\cdots dt_d\,
&=\, \int_{\Rdb^d} b_s(t_1,\ldots,t_d)U^1_{t_1-s}\cdots
U^d_{t_d-s}\, dt_1\cdots dt_d\\
&=\, U^1_{-s}\cdots U^d_{-s} \int_{\Rdb^d} 
b_s(t_1,\ldots,t_d)U^1_{t_1}\cdots
U^d_{t_d}\, dt_1\cdots dt_d\,.
\end{align*}
Hence 
$$
\Bignorm{\int_{\Rdb^d} b(t_1,\ldots,t_d)U^1_{t_1}\cdots
U^d_{t_d}\, dt_1\cdots dt_d}\leq\, 
K_1\cdots K_d\,\Bignorm{\int_{\Rdb^d} b_s(t_1,\ldots,t_d)U^1_{t_1}\cdots
U^d_{t_d}\, dt_1\cdots dt_d}
$$
Moreover $\norm{\widehat{b}}_{\infty,\Rdb^d}=
\norm{\widehat{b_s}}_{\infty,\Rdb^d}$ and 
$b_s$ has support in $[0,\infty)^d$, that is, $b_s\in L^1(\Rdb_+^d)$.
Hence to prove (ii), it sufices to establish an estimate 
(\ref{HP}) for 
\begin{equation}\label{+d}
b\in L^1(\Rdb_+)\otimes\cdots \otimes L^1(\Rdb_+).
\end{equation}
If, for any $b$ as above, $L_b$ had an extension to an element of 
$H^\infty_{0,1}(\Sigma_\theta^d)$
for some $\theta>\frac{\pi}{2}$,  then the argument in the proof of
``(ii)$\,\Rightarrow\,$(i)" would provide an estimate  (\ref{HP}).
This does not hold true in general and the rest of the proof
is just a way to circumvent that.

Following \cite[Section 2.2]{Haa} we let 
$$
\E(\Sigma_\theta)=H^\infty_0(\Sigma_\theta)\oplus{\rm Span}\{1, R\}\,
\subset \, H^\infty(\Sigma_\theta),
$$
where $R$ is defined by $R(z)=(1+z)^{-1}$.
Then for any $\theta\in\bigl(\frac{\pi}{2},\pi\bigr)$, 
holomorphic functional calculus associated with
$(B_1,\ldots,B_d)$ can be naturally defined on
$$
\E(\Sigma_\theta)\otimes\cdots\otimes 
\E(\Sigma_\theta)\,\subset\, H^\infty(\Sigma_{\theta}^d).
$$
It is easy to check (left to the reader) that the assumption (i) implies the 
existence of a constant $K\geq 1$, not depending
on $\theta$, such that
\begin{equation}\label{FC-E}
\norm{F(B_1,\ldots,B_d)}\leq\,K
\norm{F}_{\infty,\Sigma_{\frac{\pi}{2}}^d},
\qquad F\in \E(\Sigma_\theta)\otimes\cdots\otimes 
\E(\Sigma_\theta).
\end{equation}
For any $\varepsilon>0$, we set
$h_\varepsilon(z)= \varepsilon + z(1+\varepsilon z)^{-1}$ and we define
$$
B_{k,\varepsilon} = h_\varepsilon(B_k)
=\, \varepsilon I_Y+ B_k(I_Y+\varepsilon B_k)^{-1},
\qquad k=1,\ldots,d.
$$
Each of these operators is bounded and invertible, 
with spectrum included in
$\Sigma_{\frac{\pi}{2}}$.
For convenience we set 
$$
H_\varepsilon(z_1,\ldots,z_d) = \bigl(h_\varepsilon(z_1),
\ldots, h_\varepsilon(z_d)\bigr),\qquad
(z_1,\ldots,z_d)\in\Sigma_{\frac{\pi}{2}}^d.
$$
For any $a\in L^1(\Rdb_+)$, $L_a\circ h_\varepsilon$
extends to an element of $H^\infty(\Sigma_\theta)$
for some $\theta\in\bigl(\frac{\pi}{2},\pi\bigr)$
depending on $\varepsilon$. This extension is holomorphic
at $0$ and $\infty$, hence belongs to  
$\E(\Sigma_\theta)$.
Further by the
composition rule (see \cite[Section 2.4]{Haa}),
$L_a\circ h_\varepsilon(B_k) = L_a(B_{k,\varepsilon})$
for any $k=1,\ldots,d$.

This readily implies that for any
$b$ as in (\ref{+d}), we 
have
$$
\norm{L_b\circ H_\varepsilon(B_1,\ldots,B_d)}\
\leq K\norm{L_b\circ H_\varepsilon}_{\infty,\Sigma_{\frac{\pi}{2}}^d}
\leq K \norm{\widehat{b}}_{\infty,\Rdb^d},
$$
by (\ref{FC-E}), and 
$$
L_b\bigl(B_{1,\varepsilon},\ldots, B_{d,\varepsilon}\bigr)\,=\,
L_b\circ H_\varepsilon(B_1,\ldots,B_d).
$$
Now observe that 
$$
L_b\bigl(B_{1,\varepsilon},\ldots, B_{d,\varepsilon}\bigr)
\,=\,\int_{\Rdb^d} b(t_1,\ldots,t_d)\,e^{-t_1 B_{1,\varepsilon}}\cdots
e^{-t_d B_{d,\varepsilon}}\, dt_1\ldots dt_d\,.
$$
Consequently,
$$
\Bignorm{\int_{\Rdb^d} b(t_1,\ldots,t_d)\,e^{-t_1 B_{1,\varepsilon}}
\cdots
e^{-t_d B_{d,\varepsilon}}\, dt_1 \ldots dt_d}\,\leq\, K 
\norm{\widehat{b}}_{\infty,\Rdb^d}
$$
for any $\varepsilon>0$.

For any $k=1,\ldots, d$, the semigroups
$(e^{-tB_{k,\varepsilon}})_{t\geq 0}$ are uniformly bounded and
for any $t\geq 0$, $e^{-t B_{k,\varepsilon}}\to U_t^k$
strongly, when $\varepsilon \to 0$. This follows from the classical
Yosida approximation principle. Hence using Lebesgue's dominated 
convergence Theorem, the above inequality implies
an estimate 
(\ref{HP}) for $b$ satisfying (\ref{+d}).
\end{proof}

\bigskip
\noindent
{\bf Acknowledgements.} 
The two authors were supported by the French 
``Investissements d'Avenir" program, 
project ISITE-BFC (contract ANR-15-IDEX-03).

\medskip

\end{document}